\documentclass[12pt,oneside]{amsart}  
\usepackage{amssymb}
\usepackage{amsfonts}
\usepackage{amsthm}
\usepackage{amsthm}
\usepackage{amscd}
\numberwithin{equation}{section}
\theoremstyle{plain}
\newtheorem{theorem}{Theorem}[section]

\newtheorem{proposition}[theorem]{Proposition}
\newtheorem{corollary}[theorem]{Corollary}
\newtheorem{lemma}[theorem]{Lemma}
\theoremstyle{definition}
\newtheorem{definition}[theorem]{Definition}
\theoremstyle{remark}
\newtheorem*{remark}{Remark}

\begin{document}
%
%

\newcommand{\M}{\mathcal{M}_{g,N+1}^{(1)}}
\newcommand{\Teich}{\mathcal{T}_{g,N+1}^{(1)}}
\newcommand{\T}{\mathrm{T}}
\newcommand{\corr}{\bf}
\newcommand{\vac}{|0\rangle}
\newcommand{\Ga}{\Gamma}
\newcommand{\new}{\bf}
\newcommand{\define}{\def}
\newcommand{\redefine}{\def}
\newcommand{\Cal}[1]{\mathcal{#1}}
\renewcommand{\frak}[1]{\mathfrak{{#1}}}
\newcommand{\refE}[1]{(\ref{E:#1})}
\newcommand{\refS}[1]{Section~\ref{S:#1}}
\newcommand{\refSS}[1]{Section~\ref{SS:#1}}
\newcommand{\refT}[1]{Theorem~\ref{T:#1}}
\newcommand{\refO}[1]{Observation~\ref{O:#1}}
\newcommand{\refP}[1]{Proposition~\ref{P:#1}}
\newcommand{\refD}[1]{Definition~\ref{D:#1}}
\newcommand{\refC}[1]{Corollary~\ref{C:#1}}
\newcommand{\refL}[1]{Lemma~\ref{L:#1}}
\newcommand{\R}{\ensuremath{\mathbb{R}}}
\newcommand{\C}{\ensuremath{\mathbb{C}}}
\newcommand{\N}{\ensuremath{\mathbb{N}}}
\newcommand{\Q}{\ensuremath{\mathbb{Q}}}
\renewcommand{\P}{\ensuremath{\mathbb{P}}}
\newcommand{\Z}{\ensuremath{\mathbb{Z}}}
\newcommand{\kv}{{k^{\vee}}}
\renewcommand{\l}{\lambda}
\newcommand{\gb}{\overline{\mathfrak{g}}}
\newcommand{\g}{\mathfrak{g}}
\newcommand{\gh}{\widehat{\mathfrak{g}}}
\newcommand{\ghN}{\widehat{\mathfrak{g}_{(N)}}}
\newcommand{\gbN}{\overline{\mathfrak{g}_{(N)}}}
\newcommand{\tr}{\mathrm{tr}}
\newcommand{\sln}{\mathfrak{sl}}
\newcommand{\sn}{\mathfrak{s}}
\newcommand{\so}{\mathfrak{so}}
\newcommand{\spn}{\mathfrak{sp}}
\newcommand{\gl}{\mathfrak{gl}}
\newcommand{\slnb}{{\overline{\mathfrak{sl}}}}
\newcommand{\snb}{{\overline{\mathfrak{s}}}}
\newcommand{\sob}{{\overline{\mathfrak{so}}}}
\newcommand{\spnb}{{\overline{\mathfrak{sp}}}}
\newcommand{\glb}{{\overline{\mathfrak{gl}}}}
\newcommand{\Hwft}{\mathcal{H}_{F,\tau}}
\newcommand{\Hwftm}{\mathcal{H}_{F,\tau}^{(m)}}

\newcommand{\car}{{\mathfrak{h}}}    
\newcommand{\bor}{{\mathfrak{b}}}    
\newcommand{\nil}{{\mathfrak{n}}}    
\newcommand{\vp}{{\varphi}}
\newcommand{\bh}{\widehat{\mathfrak{b}}}  
\newcommand{\bb}{\overline{\mathfrak{b}}}  
\newcommand{\Vh}{\widehat{\mathcal V}}
\newcommand{\KZ}{Kniz\-hnik-Zamo\-lod\-chi\-kov}
\newcommand{\TUY}{Tsuchia, Ueno  and Yamada}
\newcommand{\KN} {Kri\-che\-ver-Novi\-kov}
\newcommand{\pN}{\ensuremath{(P_1,P_2,\ldots,P_N)}}
\newcommand{\xN}{\ensuremath{(\xi_1,\xi_2,\ldots,\xi_N)}}
\newcommand{\lN}{\ensuremath{(\lambda_1,\lambda_2,\ldots,\lambda_N)}}
\newcommand{\iN}{\ensuremath{1,\ldots, N}}
\newcommand{\iNf}{\ensuremath{1,\ldots, N,\infty}}

\newcommand{\tb}{\tilde \beta}
\newcommand{\tk}{\tilde \kappa}
\newcommand{\ka}{\kappa}
\renewcommand{\k}{\kappa}

\newcommand{\Pif} {P_{\infty}}
\newcommand{\Pinf} {P_{\infty}}
\newcommand{\PN}{\ensuremath{\{P_1,P_2,\ldots,P_N\}}}
\newcommand{\PNi}{\ensuremath{\{P_1,P_2,\ldots,P_N,P_\infty\}}}
\newcommand{\Fln}[1][n]{F_{#1}^\lambda}
\newcommand{\tang}{\mathrm{T}}
\newcommand{\Kl}[1][\lambda]{\can^{#1}}
\newcommand{\A}{\mathcal{A}}
\newcommand{\U}{\mathcal{U}}
\newcommand{\V}{\mathcal{V}}
\renewcommand{\O}{\mathcal{O}}
\newcommand{\Ae}{\widehat{\mathcal{A}}}
\newcommand{\Ah}{\widehat{\mathcal{A}}}
\newcommand{\La}{\mathcal{L}}
\newcommand{\Le}{\widehat{\mathcal{L}}}
\newcommand{\Lh}{\widehat{\mathcal{L}}}
\newcommand{\eh}{\widehat{e}}
\newcommand{\Da}{\mathcal{D}}
\newcommand{\kndual}[2]{\langle #1,#2\rangle}
\newcommand{\cins}{\frac 1{2\pi\mathrm{i}}\int_{C_S}}
\newcommand{\cinsl}{\frac 1{24\pi\mathrm{i}}\int_{C_S}}
\newcommand{\cinc}[1]{\frac 1{2\pi\mathrm{i}}\int_{#1}}
\newcommand{\cintl}[1]{\frac 1{24\pi\mathrm{i}}\int_{#1 }}
\newcommand{\w}{\omega}
\newcommand{\ord}{\operatorname{ord}}
\newcommand{\res}{\operatorname{res}}
\newcommand{\nord}[1]{:\mkern-5mu{#1}\mkern-5mu:}
\newcommand{\Fn}[1][\lambda]{\mathcal{F}^{#1}}
\newcommand{\Fl}[1][\lambda]{\mathcal{F}^{#1}}
\renewcommand{\Re}{\mathrm{Re}}

\newcommand{\ha}{H^\alpha}

\define\ldot{\hskip 1pt.\hskip 1pt}
\define\ifft{\qquad\text{if and only if}\qquad}
\define\a{\alpha}
\redefine\d{\delta}
\define\w{\omega}
\define\ep{\epsilon}
\redefine\b{\beta}
\redefine\t{\tau}
\redefine\i{{\,\mathrm{i}}\,}
\define\ga{\gamma}
\define\cint #1{\frac 1{2\pi\i}\int_{C_{#1}}}
\define\cintta{\frac 1{2\pi\i}\int_{C_{\tau}}}
\define\cintt{\frac 1{2\pi\i}\oint_{C}}
\define\cinttp{\frac 1{2\pi\i}\int_{C_{\tau'}}}
\define\cinto{\frac 1{2\pi\i}\int_{C_{0}}}
\define\cinttt{\frac 1{24\pi\i}\int_C}
\define\cintd{\frac 1{(2\pi \i)^2}\iint\limits_{C_{\tau}\,C_{\tau'}}}
\define\cintdr{\frac 1{(2\pi \i)^3}\int_{C_{\tau}}\int_{C_{\tau'}}
\int_{C_{\tau''}}}
\define\im{\operatorname{Im}}
\define\re{\operatorname{Re}}
\define\res{\operatorname{res}}
\redefine\deg{\operatornamewithlimits{deg}}
\define\ord{\operatorname{ord}}
\define\rank{\operatorname{rank}}
\define\fpz{\frac {d }{dz}}
\define\dzl{\,{dz}^\l}
\define\pfz#1{\frac {d#1}{dz}}

\define\K{\Cal K}
\define\U{\Cal U}
\redefine\O{\Cal O}
\define\He{\text{\rm H}^1}
\redefine\H{{\mathrm{H}}}
\define\Ho{\text{\rm H}^0}
\define\A{\Cal A}
\define\Do{\Cal D^{1}}
\define\Dh{\widehat{\mathcal{D}}^{1}}
\redefine\L{\Cal L}
\redefine\D{\Cal D^{1}}
\define\KN {Kri\-che\-ver-Novi\-kov}
\define\Pif {{P_{\infty}}}
\define\Uif {{U_{\infty}}}
\define\Uifs {{U_{\infty}^*}}
\define\KM {Kac-Moody}
\define\Fln{\Cal F^\lambda_n}
\define\gb{\overline{\mathfrak{ g}}}
\define\G{\overline{\mathfrak{ g}}}
\define\Gb{\overline{\mathfrak{ g}}}
\redefine\g{\mathfrak{ g}}
\define\Gh{\widehat{\mathfrak{ g}}}
\define\gh{\widehat{\mathfrak{ g}}}
\define\Ah{\widehat{\Cal A}}
\define\Lh{\widehat{\Cal L}}
\define\Ugh{\Cal U(\Gh)}
\define\Xh{\hat X}
\define\Tld{...}
\define\iN{i=1,\ldots,N}
\define\iNi{i=1,\ldots,N,\infty}
\define\pN{p=1,\ldots,N}
\define\pNi{p=1,\ldots,N,\infty}
\define\de{\delta}

\define\kndual#1#2{\langle #1,#2\rangle}
\define \nord #1{:\mkern-5mu{#1}\mkern-5mu:}
\define \sinf{{\widehat{\sigma}}_\infty}
\define\Wt{\widetilde{W}}
\define\St{\widetilde{S}}
\define\Wn{W^{(1)}}
\define\Wtn{\widetilde{W}^{(1)}}
\define\btn{\tilde b^{(1)}}
\define\bt{\tilde b}
\define\bn{b^{(1)}}
%
\define\eps{\varepsilon}    
\define\doint{({\frac 1{2\pi\i}})^2\oint\limits _{C_0}
       \oint\limits _{C_0}}                            
\define\noint{ {\frac 1{2\pi\i}} \oint}   
\define \fh{{\frak h}}     
\define \fg{{\frak g}}     
\define \GKN{{\Cal G}}   
\define \gaff{{\hat\frak g}}   
\define\V{\Cal V}
\define \ms{{\Cal M}_{g,N}} 
\define \mse{{\Cal M}_{g,N+1}} 
\define \tOmega{\Tilde\Omega}
\define \tw{\Tilde\omega}
\define \hw{\hat\omega}
\define \s{\sigma}
\define \car{{\frak h}}    
\define \bor{{\frak b}}    
\define \nil{{\frak n}}    
\define \vp{{\varphi}}
\define\bh{\widehat{\frak b}}  
\define\bb{\overline{\frak b}}  
\define\Vh{\widehat V}
\define\KZ{Knizhnik-Zamolodchikov}
\define\ai{{\alpha(i)}}
\define\ak{{\alpha(k)}}
\define\aj{{\alpha(j)}}
\newcommand{\laxgl}{\overline{\mathfrak{gl}}}
\newcommand{\laxsl}{\overline{\mathfrak{sl}}}
\newcommand{\laxs}{\overline{\mathfrak{s}}}
\newcommand{\laxg}{\overline{\frak g}}
\newcommand{\bgl}{\laxgl(n)}
\newcommand{\tX}{\widetilde{X}}
\newcommand{\tY}{\widetilde{Y}}
\newcommand{\tZ}{\widetilde{Z}}
\vspace*{-1cm}
%
%
\hspace*{\fill} math.QA/yymmnnn
\vspace*{2cm}

\title[Central extensions of Lax operator algebras]
     {Central extensions of Lax operator algebras}
\thanks{This work was supported by the grant
R1F10L05 of the University of Luxembourg, 
the RFBR project 05-01-00170, and by the Programme
"Mathematical methods of non-linear dynamics" of the Russian Academy of
Sciences
}
\author[M. Schlichenmaier]{Martin Schlichenmaier}
\address[Martin Schlichenmaier]{Institute of Mathematics,
University of  Luxembourg,
162 A, Avenue de la Faiencerie,
L-1511 Luxembourg, Grand Duchy  of Luxembourg}
\email{martin.schlichenmaier@uni.lu}
\author[O.K. Sheinman]{Oleg K. Sheinman}
\address[ Oleg K. Sheinman]{Steklov Mathematical Institute, ul. Gubkina, 8,
Moscow, 119991, Russia
and Independent University of Moscow,
Bolshoi Vlasievskii per. 11, Moscow, Russia}
\email{sheinman@mi.ras.ru}

\begin{abstract}
Lax operator algebras were introduced by 
Kri\-che\-ver and  Sheinman as a further development of I.Krichever's theory of
Lax operators on algebraic curves.
These  are almost-graded Lie algebras of 
current type. In this article local cocycles and associated
almost-graded central extensions are classified.
It is shown that in the case that the corresponding finite-dimensional
Lie algebra is simple the two-cohomology space is one-dimensional.
An important role is played by the action of the Lie algebra of
meromorphic vector fields on the Lax operator algebra via
suitable covariant derivatives.
\end{abstract}
\date{November 22, 2007}
\maketitle
\tableofcontents
\section{Introduction}\label{S:intro}

In this article, we give a full classification of almost-graded
central extensions for a new class of one-dimensional current
algebras --- the Lax operator algebras.

Lax operator algebras are introduced by I.Krichever and one of the
authors in \cite{KSlax}. In that work, the concept of Lax
operators on algebraic curves proposed in \cite{Klax} was
generalized to $\g$-valued Lax operators where $\g$ is one of the
classical complex Lie algebras.

We would like to remind here that in \cite{Klax} the theory of
conventional Lax and zero curvature representations with a {\it
rational} spectral parameter was generalized to the case of
algebraic curves $\Sigma$ of arbitrary genus $g$. Such
representations arise in several ways in the theory of integrable
systems, c.f. \cite{rKNU} where a zero curvature representation of
the Krichever-Novikov equation is introduced, or \cite{Klax} where
a field analog of the Calogero-Moser system on an elliptic curve
is presented. The theory of Lax operators on Riemann surfaces
proposed in \cite{Klax} includes the Hamiltonian theory of Lax and
zero curvature equations, the theory of Baker-Akhieser functions,
and an approach to corresponding algebraic-geometric solutions.

The concept of Lax operators on algebraic curves is closely
related to A.Tyurin results on the classification of
holomorphic vector bundles on algebraic curves \cite{Tyvb}. It
uses {\it Tyurin data} modelled on {\it Tyurin parameters} of such
bundles consisting of points $\ga_s$ ($s=1,\ldots ,ng$), and
associated elements $\a_s\in\C P^n$ (where $g$ denotes the genus
of the Riemann surface $\Sigma$, and $n$ corresponds to the rank of
the  bundle).

The linear space of Lax operators associated with a positive
divisor $D=\sum_k m_kP_k$, $P_k\in \Sigma$ is defined in
\cite{Klax} as the space of meromorphic $(n\times n)$
matrix-valued functions on $\Sigma$ having poles of multiplicity
at most $m_k$ at the points $P_k$, and at most simple poles at
$\gamma_s$'s. The coefficients of the Laurent expansion of those
matrix-valued functions in the neighborhood of a point $\gamma_s$
have to obey certain linear constraints parameterized by $\alpha_s$
(see relations \refE{gldef} below).

The observation that Lax operators having poles of arbitrary
orders at the points $P_k$ form an algebra with respect to the
usual point-wise multiplication became a starting point of the
considerations in \cite{KSlax}. There, for $\g=\sln(n),\ \so(n) ,\
\spn(2n)$ over $\C$, the $\g$-valued Lax operators were
introduced. The space of such operators form a Lie algebra with
respect to the point-wise bracket. We denote this algebra
by~$\gb$. Considering $\g$-valued Lax operators requires certain
modifications of the above mentioned linear constraints. It even
turned out  that for $\g=\spn(2n)$ the orders of poles at
$\ga_s$'s must be set equal to $2$. There is no doubt that by
means of appropriate modifications it is possible to construct Lax
operator algebras for other classical Lie algebras.

On the other hand, in case of absence of points $\ga_s$ (which
corresponds to trivial vector bundles) we return to the known
class of Krichever-Novikov algebras (see \cite{ShN65} for a
review). If, in addition, the genus of $\Sigma$ is equal to $0$,
and $D$ is supported at two points, we obtain (up to isomorphism)
the loop algebras.

Likewise Krichever-Novikov algebras, the Lax operator algebras
possess an almost-graded structure generalizing the graded
structure of the classic affine algebras. Recall that a Lie
algebra $\V$ is called  {\it almost-graded} if $\V=\oplus_i\V_i$
where $\dim\,\V_i<\infty$ and
$[\V_i,\V_j]\subseteq\oplus_{k=i+j-k_0}^{k=i+j+k_1}\V_k$ where
$k_0$ and $k_1$ do not depend on $i$, $j$.

The general notion of almost-graded algebras and modules over them
was introduced in \cite{KNFa}-\cite{KNFc} where the
generalizations of Heisenberg  and Virasoro algebras were
considered.
The almost-graded structure is important in the theory of
highest-weight-like representations (physically --- in
second quantization).

By one-dimensional central extensions  quantum theory enters
Lie algebra theory. A prominent example is given by the
Heisenberg algebra. The
mathematical  relevance of  central extensions is  well-known.

The equivalence classes of one-dimensional central extensions of a
Lie algebra $\V$ are in one-to-one correspondence with the
elements of $\H^2(\V,\C)$, the second  Lie algebra cohomology with
coefficients in the trivial module. In particular, a central
extension is explicitly given by a 2-cocycle of $\V$. If $\dim
\H^2(\V,\C)=1$ then there is (up to rescaling of the central
element and equivalence) only one non-trivial central extension.
By abuse of language we say that the central extension is unique.

Lax operator algebras belong to the class of one-dimensional
current algebras since their elements are meromorphic $\g$-valued
functions on Riemann surfaces. The algebras of that class having
been classically considered are graded. The problem of classifying
their central extensions was considered in a series of articles.
Here we quote only three of them: V.Kac \cite{Kac68} and R.Moody
\cite{Moody69} constructed central extensions using canonical
generators and Cartan-Serre relations; H.Garland \cite{Gar} proved
the uniqueness theorem for loop algebras with simple $\g$. For
further references see \cite[Comments to Chapter 7]{KacB}.

For the more general case of a Lie algebra of the form $\g\otimes
\A$ with an associative algebra $\A$ and a simple Lie algebra $\g$,
Ch.Kassel \cite{Kass} showed that the universal central extension is
parameterized by K\"ahler differentials modulo exact differentials.
In particular, it is not necessarily one-dimensional. Hence in
general one-dimensional central extensions are not unique.

A special case is given by  the higher genus multi-point current
algebras \cite{Shea}, \cite{Sha}, \cite{SSS}, 
\cite{Saff}. They consist of $\g$-valued meromorphic
functions on the Riemann surface with poles only at a finite
number of fixed points. In higher genus and in the multi-point
case in genus zero the central extensions  are essentially
non-unique. In fact for a simple $\g$ they are in one-to-one
correspondence with the elements of $\H_1(\Sigma\setminus
supp(D),\R)$.

We like to point out, that Kassel's result is not applicable to
Lax operator algebras as they  do not admit any factorization as
tensor product.

Coming from the applications (e.g. from second quantization) an
important role is played by {\it almost-graded} central
extensions, i.e. central extensions in the category of
almost-graded Lie algebras
\cite{Shf}, \cite{ShMMJ}, \cite{SDiss}.
Almost-graded central extensions are
given by local 2-cocycles. A 2-cocycle $\ga$ of an almost-graded
Lie algebra $\V$ is called local if there exists a $K\in\Z$  such
that $\ga(\V_i,\V_j)=0$ for $|i+j|>K$. This notion of a local
cocycle is introduced in \cite{KNFa}. A cohomology class is called
local if it contains a local representing cocycle. For a
Krichever-Novikov algebra with a simple $\g$ the almost-gradedness
implies the uniqueness of a central extension \cite{Saff}. A
similar statement was previously conjectured for Virasoro-type
algebras in \cite{KNFa}, where also the outline of a proof was
given. A complete classification of almost-graded central
extensions for Krichever-Novikov current and vector field algebras
is given by one of the authors in \cite{Scocyc,Saff}.


In this article  we solve the corresponding problem for the
Lax operator algebras
$\gb$. Here we only consider the two-point case, 
i.e. $D=P_++P_-$.
The principal structure of the multi-point case is similar and
will be considered in \cite{SSlaxm}. Again, if $\g$ is
a classical simple Lie algebra it turns out that $\gb$  has a 
unique almost-graded central  extension.


Let us describe  the content and the obtained results 
of the present article in more
detail. Let  $\L$ be the Lie algebra consisting of those
meromorphic vector fields  on $\Sigma$ which are holomorphic
outside of $\{P_+,P_-\}$. In \refS{algebras} we introduce an
$\L$-action on the Lax operator algebra $\gb$. For that we make
use of the connections $\nabla^{(\w)}$  introduced in \cite{Kiso}.
These connections again have prescribed behavior at the points of
weak singularities and are holomorphic outside of those and of 
$\{P_+,P_-\}$. Indeed, we might even require (and do so) that they
are holomorphic at $P_+$. It turns out that the Lax operator
algebra is  an almost-graded module over the algebra
consisting of those meromorphic differential operators  of degree
$\le 1$ which are holomorphic outside  $\{P_+,P_-\}$.

The $\L$-module structure, given by a choice of a connection
$\nabla^{(\w)}$, enables us to introduce below an important notion
of $\L$-invariant cocycles.

In \refS{cocylces} we introduce the following cocycles
\begin{align}\label{E:ig1}
\ga_{1,\w,C}(L,L')&= \cinc{C} \tr(L\cdot \nabla^{(\w)}L'),
\\
\label{E:ig2} \ga_{2,\w,C}(L,L')&= \cinc{C} \tr(L)\cdot
\tr(\nabla^{(\w)}L'),
\end{align}
called {\it geometric cocycles}. Here $C$ is an arbitrary cycle on
$\Sigma$ avoiding the points of possible singularities.
In another form  the
cocycles of type \refE{ig1}  were introduced in
\cite{KSlax}. We show that the corresponding cohomology classes do
not depend on the choice of the connection.

A cocycle $\gamma$ is
called $\L$-invariant if
\begin{equation}
\gamma(\nabla^{(\w)}_{e}L,L')+ \gamma(L,\nabla^{(\w)}_{e}L')=0,
\end{equation}
for all vector fields $e\in\L$. It turns out that the cocycles
\refE{ig1} and \refE{ig2} are $\L$-invariant. We call a cohomology
class $\L$-invariant if it has a representative which is
$\L$-invariant. In the case of a simple Lie algebra $\g$ the
notion of $\L$-invariance allows us to single out  a unique
element in the cohomology class. Moreover, in the $\glb(n)$ case it
is necessary to exclude nontrivial cocycles coming from the
finite-dimensional Lie algebra.

Besides those aspects, the $\L$-invariance of a cocycle is related to the
property that it  comes from a cocycle of the differential
operator algebra associated to $\gb$. See \refS{remarks} for more
information.

Again, here we are only interested in almost-graded central
extensions, hence in local cocycles, resp. cohomology classes. For
a general cycle $C$ in \refE{ig1} and \refE{ig2} neither the
cocycle nor its cohomology class is local. But if $C$ is a circle
around $P_+$ the cocycle is local, see \refP{locg1}.

Our main result is  \refT{main} which gives the following
classification. For $\slnb(n)$, $\sob(n)$, and $\spnb(2n)$ the space
of local cohomology classes is 1-dimensional. Furthermore, in
every local cohomology class there is a unique $\L$-invariant
representative. It is  given as a multiple of the cocycle
\refE{ig1} (with $C$ a circle around $P_+$). For $\glb(n)$ we
obtain, that the space of cohomology classes which are local and
having been restricted to the scalar algebra are $\L$-invariant is
two-dimensional. Furthermore, every local  and $\L$-invariant
cocycle is a linear combination of \refE{ig1} and \refE{ig2} (with
$C$ a circle around $P_+$).

The proofs are presented in \refS{induction} and \refS{direct}. We
follow the general strategy developed in \cite{Scocyc} and adapt
it to our more general situation. 
In Section 4, 
using  the  locality and
$\L$-invariance, we  show that the cocycle is given by its values
at the pairs of homogenous elements for which the sum of their
degrees is equal to zero. Furthermore, we show that an
$\L$-invariant and local cocycle will be uniquely fixed by a
certain finite number of such cocycle values. A more detailed
analysis shows that the cocycles are of the form introduced above.

In \refS{direct} we show the following: Let $\g$ be a simple
finite-dimensional Lie algebra and $\gb$ any associated two-point
algebra of {\it current type}, e.g. a Lax operator algebra, a
Krichever-Novikov current algebra $\g\otimes \A$, or a Kac-Moody
current algebra $\g\otimes\C[z,z^{-1}]$, then every local cocycle
is cohomologous to a cocycle which is fixed by its value at 
one special pair of elements in $\gb$ (i.e. by
$\ga(H^\a_1,H^{\a}_{-1})$ for one fixed simple root $\a$, see
\refS{direct} for the notation). Hence in these cases the cohomology
spaces can be at most 1-dimensional. Combining this with the
fact of existence of the cocycle \refE{ig1} we obtain the
uniqueness and existence of the local cohomology class.
Furthermore, up to rescaling \refE{ig1} is the unique
$\L$-invariant and local cocycle.

We substantially
 use the internal structure of the  Lie algebra $\gb$ related
to the root system of the underlying finite dimensional
simple Lie algebra $\g$, and
the almost-gradedness of $\gb$. Recall that in the classical case
$\g\otimes\C[z,z^{-1}]$ the algebra is graded. In this very
special case the  chain of arguments gets simpler and is similar
to the arguments of  Garland \cite{Gar}.

The presented arguments remain valid in a more general context, as
one only refers to the internal structure of $\g$, the
almost-gradedness of $\gb$, and the $\L$-invariance, see the remark
at the end of \refS{direct}.

\medskip

By adapting the techniques in \cite{Saff}, the corresponding
uniqueness and classification results can be obtained for the case
of more than two points allowed for  ``strong'' singularities.
More precisely, let
$$
I:=\{P_1,P_2,\ldots,P_K\} \qquad O:=\{Q_1,Q_2,\ldots,Q_L\}
$$
be two non-empty disjoint subsets of points on $\Sigma$. This is
the same set-up as for the multi-point algebras of
Krichever-Novikov type as introduced and studied in
\cite{SLa,SLb,SLc,SDiss,Scocyc,Saff}. In the definition of the Lax
operators now the elements are allowed to
have poles  at the points of $I\cup O$. The splitting into these
subsets defines an almost-grading of the corresponding algebras.
It can be shown that for the simple Lie algebra case the space of
cohomology classes which are bounded from above (i.e. 
those which vanish 
if  evaluated for pairs of homogenous elements with sum of
degrees above a uniform threshold) is $K$-dimensional ($K=\#I$).
In the two-point case  every bounded cocycle is  local. This is
not the case here. By techniques similar to \cite{Saff} it turns
out that up to rescaling there is a unique cohomology class which is local. A
corresponding result is true for $\glb(n)$, i.e., the
space of local and $\L$-invariant cohomology classes will be 
two-dimensional. Details will appear in a forthcoming paper
\cite{SSlaxm}.
\newpage
\section{The algebras and their almost-grading}\label{S:algebras}
\subsection{The algebras}\label{S:alg}
$ $

Let $\Sigma$ be a compact Riemann surface of genus $g$ with 
two marked points $P_+$ and $P_-$. For $n\in\N$ we fix
$n\cdot g$ additional  points 
\begin{equation}
W:=\{\ga_s\in\Sigma\setminus\{P_+,P_-\}\mid s=1,\ldots, n g\}.
\end{equation}
To every point $\ga_s$ we assign a vector $\a_s\in\C^n$.
The system 
\begin{equation}
T:=\{(\ga_s,\a_s)\in\Sigma\times \C^n\mid s=1,\ldots, n g\}
\end{equation}
is called 
{\it Tyurin data} below. 
This data is related to
the moduli of  vector bundles over $\Sigma$.
In particular, for generic values of $(\ga_s,\a_s)$ with 
$\a_s\ne 0$ the tuples of pairs 
 $(\ga_s,[\a_s])$ with $[\a_s]\in\P^{n-1}(\C)$ 
parameterize semi-stable rank $n$ degree $n g$ framed holomorphic 
vector bundles over $\Sigma$, see \cite{Tyvb}.

We fix local coordinates
$z_\pm$ at $P_\pm$ and $z_s$ at $\ga_s$,  $s=1,\ldots, n g$. 
In the following let $\g$ be  one of the matrix algebras
$\gl(n)$, $\sln(n)$, $\so(n)$, $\spn(2n)$, or $\sn(n)$, where
the latter  denotes the algebra of scalar matrices.

We will consider meromorphic functions
\begin{equation}
L:\ \Sigma\ \to\  \g,
\end{equation}
which are
holomorphic outside  $W\cup \{P_+, P_-\}$, have at most poles
of order one (resp. of order two for $\spn(2n)$) 
at the points in $W$, and fulfill certain 
conditions at $W$ depending on $T$ and $\g$.
The singularities at $W$ are called {\it weak singularities}.
These objects were introduced by Krichever  \cite{Klax}  for $\gl(n)$ 
in the context of Lax operators for algebraic curves,
and further generalized in \cite{KSlax}.
In particular, the additional requirements for the expansion
at $W$ we  give now were introduced there.

\medskip
\noindent
The above mentioned conditions for {\bf $\gl(n)$} 
are as follows. Let $T$ be fixed.
For $s=1,\ldots, ng$ we require that there exist $\b_s\in\C^n$ 
and $\ka_s\in \C$ such that the
function $L$ has the following expansion at $\ga_s\in W$ 
\begin{equation}\label{E:glexp}
L(z_s)=\frac {L_{s,-1}}{z_s}+
L_{s,0}+\sum_{k>0}L_{s,k}{z_s^k}
\end{equation}
with
\begin{equation}\label{E:gldef}
L_{s,-1}=\a_s \b_s^{t},\quad
\tr(L_{s,-1})=\b_s^t \a_s=0,
\quad
L_{s,0}\,\a_s=\ka_s\a_s.
\end{equation}
In particular, $L_{s,-1}$ is a rank 1 matrix, and if 
$\a_s\ne 0$  then it is  
an eigenvector of $L_{s,0}$.
In \cite{KSlax} it is shown that
the requirements \refE{gldef} are independent of the chosen
coordinates $z_s$ and that 
the set of all such functions constitute an associative algebra under
the point-wise matrix multiplication. 
We denote it by $\glb(n)$.

The algebra $\glb(n)$ depends both on the choice of the Tyurin parameters and
of the two points $P_+$ and $P_-$. Nevertheless we omit 
this dependence in the notation.
In view of the above relation to the moduli space of vector bundles 
note that for $\lambda_s\in\C^*$ the values $\alpha_s'=\lambda_s \alpha_s$ 
will define the same algebra as the values $\alpha_s$.

The constraints \refE{glexp} and \refE{gldef} at $W$
imply that the elements of the 
Lax operator algebra can be considered as sections of
the endomorphism bundle $End(B)$, where $B$ is the vector bundle 
corresponding to the Tyurin data.

\medskip

The splitting $\gl(n)=\sn(n)\oplus \sln(n)$ given by
\begin{equation}
X\mapsto \left(\ \frac {\tr(X)}{n}I_n\ ,\ X-\frac {\tr(X)}{n}I_n\ \right),
\end{equation}
where $I_n$ is the $n\times n$-unit matrix,
induces a corresponding splitting for the  Lax operator 
algebra  $\glb(n)$:
\begin{equation}
 \glb(n)=\snb(n)\oplus \slnb(n).
\end{equation}
For {\bf $\slnb(n)$} the only additional condition  is that 
in \refE{glexp} all matrices $L_{s,k}$ are  trace-less.
The condition \refE{gldef}  remains unchanged.

For {\bf $\snb(n)$} all matrices in \refE{glexp}
are scalar  matrices. This implies that
the corresponding 
$L_{s,-1}$  vanish. In particular, the elements
of $\snb(n)$ are holomorphic at $W$.
Also $L_{s,0}$, as a scalar matrix, has every $\a_s$ 
as eigenvector.
This means that beside the holomorphicity there are no
further conditions.

\medskip

In the case of {\bf $\so(n)$} we require that
all $L_{s,k}$ in \refE{glexp} are  skew-symmetric.
In particular, they are trace-less.
The set-up has to be slightly modified following \cite{KSlax}.
First only  those Tyurin parameters $\a_s$ are allowed which satisfy
$\a_s^t\a_s=0$.
Then, \refE{gldef} is modified in  the following way:
\begin{equation}\label{E:sodef}
L_{s,-1}=\a_s\b_s^t-\b_s\a_s^t,
\quad
\tr(L_{s,-1})=\b_s^t\a_s=0,
\quad
L_{s,0}\,\a_s=\ka_s\a_s.
\end{equation}
Again \refE{sodef} does not depend on the coordinates $z_s$ and 
under the point-wise
matrix commutator the set of such maps constitute a Lie algebra,
see  \cite{KSlax}.

\medskip
For {\bf $\spn(2n)$}
we consider  a symplectic form  $\hat\sigma$  
for $\C^{2n}$ given by
a non-degenerate skew-symmetric matrix $\sigma$.
Without loss of generality we might even assume 
that this matrix is given in 
the 
standard form
$\sigma=\begin{pmatrix} 0& I_n
\\ -I_n&0
\end{pmatrix}
$.
The Lie algebra $\spn(2n)$ is the Lie algebra of
matrices $X$ such that $X^t\sigma+\sigma X=0$.
This is equivalent to $X^t=-\sigma X\sigma^{-1}$, which implies that
$\tr(X)=0$.
For the standard form above, 
$X\in\spn(2n) $ if and only if
\begin{equation}
X=\begin{pmatrix} A&B
\\
C&-A^t
\end{pmatrix}, \qquad B^t=B,\quad C^t=C.
\end{equation}
At the weak singularities we have the expansion
\begin{equation}\label{E:glexpsp}
L(z_s)=\frac {L_{s,-2}}{z_s^2}+\frac {L_{s,-1}}{z_s}+
L_{s,0}+L_{s,1}{z_s}+\sum_{k>1}L_{s,k}{z_s^k}.
\end{equation}
The condition \refE{gldef} is  modified as
follows (see \cite{KSlax}):
there exist $\b_s\in\C^{2n}$, 
$\nu_s,\ka_s\in\C$ such that 
\begin{equation}\label{E:spdef}
L_{s,-2}=\nu_s \a_s\a_s^t\sigma,\quad
L_{s,-1}=(\a_s\b_s^t+\b_s\a_s^t)\sigma,
\quad{\b_s}^t\sigma\a_s=0,\quad
L_{s,0}\,\a_s=\kappa_s\a_s.
\end{equation}
{}Moreover, we require
\begin{equation}
\a_s^t\sigma L_{s,1}\a_s=0.
\end{equation}
Again in \cite{KSlax} it is shown that under the point-wise
matrix commutator the set of such maps constitute a Lie algebra.

\medskip

We summarize
\begin{theorem}[\cite{KSlax}]
The space $\gb$ of Lax operators  is a Lie algebra under the
point-wise matrix commutator. For $\gb=\glb(n)$ it is an associative 
algebra under point-wise matrix multiplication.
\end{theorem}
These Lie algebras are called {\it Lax operator algebras}.

If we take  $\a_s=0\in\C^n$ (resp. $\in\C^{2n}$) as Tyurin parameter
then there 
will be no weak singularities. In this way the usual two-point
Krichever-Novikov current algebras 
$\gb=\g\otimes\A$ are obtained \cite{Sha}.
Here $\A$ is the algebra of meromorphic functions on $\Sigma$ holomorphic
outside $P_\pm$ (see below).
{}From this point of view the Lax operator algebras might be also
called {\it generalized Krichever-Novikov current algebras}.

As noticed  above, for $\snb(n)$ there are  no weak singularities
and there are  no conditions  for the constant term.
Hence  $\snb(n)$  coincides with the Krichever-Novikov function
algebra, i.e. 
\begin{equation}\label{E:sint}
\snb(n)\cong \sn(n)\otimes \A\cong \A,
\end{equation}
as associative algebras.

Note also that if in addition the
genus is equal to zero, the
Lax operator  algebras give the
classical Kac-Moody current algebras.

\bigskip
\subsection{The almost-graded structure}
$ $

\medskip
By means of the power series
expansions at the points $P_+$ and $P_-$ 
we are able to introduce an almost-grading, as it is done for the
Krichever-Novikov current algebras, \cite{SSS}, \cite{Saff}.

To write down explicitly the conditions we have to restrict ourselves
with the case when all our marked points (including the points in $W$) 
are in generic position.
Let $\gb$ be one of the Lax operator  algebras introduced above.
For $m\le -g-1$ or $m\ge 1$ we consider the subspace
\begin{multline}\label{E:almdeg}
\gb_m:=\{L\in\gb\mid
\exists X_+,X_-\in\g \quad \text{ such that }
\\
L(z_+)=X_+z^m_++O(z_+^{m+1}),\ 
L(z_-)=X_-z^{-m-g}_-+O(z_-^{-m-g+1})\}.
\end{multline}
For $\g$ semi-simple 
and \ $\{\gamma_s\in W\mid \a_s\ne 0\}\ne\emptyset\ $\  this
definition works also for the other values of $m$.
If $\g$ is equal to $\gl(n)$  or $\sn(n)$  then in  the cases
$-g\le m\le 0$
the conditions at $P_-$ have to be slightly modified 
\cite{KSlax}. 
In fact, we  take  
$\gl(n)_m=\sln(n)_m\oplus\sn(n)_m$ 
and use for $\sln(n)$ the grading introduced above and for
$\sn(n)\cong\A$ the grading of $\A$, which we recall in
\refS{Amod}, see also \cite{KNFa}.
If $\{\gamma_s\in W\mid \a_s\ne 0\}=\emptyset$ then $\gb=\g\otimes \A$
and the grading comes from the grading of  $\A$,
see \cite{SSS}.

We call the $\gb_m$ the homogenous subspaces of degree $m$
in $\gb$.
\begin{theorem} \cite{KSlax}\label{T:almgrad}
The Lie algebras $\gb$ are almost-graded algebras 
with respect to the degree given by the  $\gb_m$'s.
More precisely, 

\noindent
(1) $\dim \gb_m=\dim \g$,

\noindent
(2)
$\gb=\bigoplus\limits_{m\in\Z}^{\hphantom{A}}\gb_m$ 

\noindent
(3) 
there exist a constant $M$ such that 
\begin{equation}\label{E:alm}
[\gb_m,\gb_k]\subseteq \bigoplus_{h=m+k}^{m+k+M}\gb_h.
\end{equation}
\end{theorem}
In \cite{KSlax}, it is found that if $\g=\sln(n),\spn(2n),\so(n)$ 
then $M=g$. We do not need it in the following.

\begin{remark}
The result about the almost-grading is also true if the points 
$P_+$, $P_-$ and $W$ are
not in generic position. In this case the requirement for the orders
at the point $P_-$ has to be adapted.
\end{remark}
\begin{proposition}\label{P:locex}
Let $X$ be an element of $\g$. For each $m$  there is
a unique element $X_m$ in $\gb_m$ such that 
\begin{equation}\label{E:locex}
X_m=
Xz_+^m+O(z_+^{m+1}).
\end{equation}
\end{proposition}
\begin{proof}
{}From the first statement of \refT{almgrad}, i.e. that
 $\dim \gb_m=\dim \g$ 
it follows that there is a unique combination of the basis elements
such that \refE{locex} is true.
\end{proof}
Given $X\in\g$, by $X_m$  we denote the unique element
in $\gb_m$ defined via \refP{locex}.

Sometimes it will be useful to consider also the
induced filtration
\begin{equation}
F_k:=\bigoplus_{m\ge k} \laxg_m,\qquad
F_k\subseteq F_{k'},\ k\ge k',\qquad
[F_k,F_m]\subseteq  F_{k+m}.
\end{equation}

The result \refE{alm} can be strengthen in the following way
\begin{proposition}
Let $X_k$ and $Y_m$ be the elements in $\gb_k$ and $\gb_m$ 
corresponding to $X,Y\in\g$ respectively then 
\begin{equation}\label{E:alalg}
[X_k,Y_m]={[X,Y]}_{k+m}+L, 
\end{equation}
with $[X,Y]$ the bracket in $\g$ and $L\in F_{k+m+1}$.
\end{proposition}
\begin{proof}
Using for $X_k$ and $Y_m$ the expression \refE{locex} we obtain
$$
[X_k,Y_m]=[X,Y]z_+^{k+m}+O(z_+^{k+m+1}).
$$
Hence,
$$
[X_k,Y_m]-([X,Y])_{k+m}=O(z_+^{k+m+1})\in F_{k+m+1},
$$
which is the claim.
\end{proof}
\begin{lemma}\label{L:weak}
Let $\g$ be simple and $y\in\gb$ then for every $m\in\Z$ there
exists finitely many elements 
$y^{(i,1)}, y^{(i,2)}\in\gb$, $i=1,\ldots, l=l(m)$ such that
\begin{equation}
y-\sum_{i=1}^l\;[y^{(i,1)}, y^{(i,2)}]\quad\in\quad F_m.
\end{equation}
\end{lemma}
\begin{proof}
If the expansion of $y$ at $P_+$ starts with order $k$ then
$y=X_k+y'$ with $y'\in F_{k+1}$,
$X\in\g$ and $X_k$ is
the corresponding element of degree $k$.
As $\g$ is simple it is perfect, hence there exist 
$X^{(1)},X^{(2)}\in\g$ such that 
$X=[X^{(1)},X^{(2)}]$. This implies
\begin{equation}
X_k=[X^{(1)}_0,X^{(2)}_k]+y'',\ \text{with } y'' \in F_{k+1},\quad
\text{or }
y=[X^{(1)}_0,X^{(2)}_k]+(y'+y'').
\end{equation} 
Using the same argument for $(y+y'')\in F_{k+1}$ the claim follows
by induction.
\end{proof}
This lemma might be considered as {\it weak perfectness} for the 
Lax operator algebras.
Note that  the 
usual Krichever-Novikov current algebras $\gb$ for $\g$ simple 
are perfect, see \cite[Prop. 3.2]{Saff}.

\subsection{Module structure over $\A$ and $\L$}\label{S:Amod}
$ $

In the following we recall the definitions of the
Krichever-Novikov function
algebra $\A$ and of the Krichever-Novikov vector field algebra $\L$.
Let $\A$ respectively $\L$ be the space of meromorphic functions
respectively meromorphic vector fields 
on $\Sigma$, holomorphic on  $\Sigma\setminus \{P_+,P_-\}$.
In particular, they are holomorphic also at the points in $W$.
Obviously, $\A$ is an associative algebra under the point-wise
product and $\L$ is a Lie algebra under the Lie bracket of 
vector fields.
By exhibiting special basis elements \cite{KNFa} these algebras are
endowed with an almost-graded structure.

In the case of $\A$ we denote the basis by 
$\{A_m\mid m\in\Z\}$. The $A_m$ are given by the requirement
that $\ord_{P_+}(A_m)=m$ and a complementary requirement at
$P_-$ to fix $A_m$ up to a scalar multiple uniquely. 
For a generic $m$ and the points $P_+$ and $P_-$  in generic position
this requirement is 
  $\ord_{P_-}(A_m)=-m-g$.
To fix the scalar multiple we require that
locally at $P_+$, with respect to the chosen local coordinate $z_+$,
we have the expansion
\begin{equation}
A_m(z_+)=z_+^m+O(z_+^{m+1}).
\end{equation}
Based on these elements we set $\A_m=\langle A_m\rangle $ and obtain
the almost-graded (associative) algebra structure
\begin{equation}
\A=\bigoplus_{m\in\Z}\A_m,\qquad
\A_k\cdot \A_m\subseteq
\bigoplus_{h=k+m}^{k+m+M_1}\A_h,
\end{equation}
with a constant $M_1$ not depending on $k$ and $m$.
Moreover
\begin{equation}
A_k\cdot A_m=A_{k+m}+\sum_{h=k+m+1}^{k+m+M_1}\a_{k,m}^h
A_h,\quad  \a_{k,m}^h\in\C.
\end{equation}

\medskip
The vector field algebra $\L$ is defined in a similar manner.
Here the basis is  $\{e_m\mid m\in\Z\}$ with
the requirement that $\ord_{P_+}(e_m)=m+1$, corresponding
orders at $P_-$ (for generic choices $\ord_{P_-}(e_m)=-m-3g-3$)
and locally at $P_+$ the expansion 
\begin{equation}
e_m(z_+)=\left(z_+^{m+1}+O(z_+^{m+2})\right)\frac {d}{dz_+}.
\end{equation}
We put $\L_m=\langle e_m\rangle$ and obtain the almost-graded structure
\begin{equation}
\L=\bigoplus_{m\in\Z} \L_m,\qquad
[\L_k,\L_m]\subseteq \bigoplus_{h=k+m}^{k+m+M_2}\L_h,
\end{equation}
with a constant $M_2$ not depending on $k$ and $m$. We obtain
\begin{equation}
[e_k,e_m]=(m-k)\,e_{k+m}+\sum_{h=k+m+1}^{k+m+M_2}\b_{k,m}^h
e_h,\quad  \b_{k,m}^h\in\C.
\end{equation}

The elements of the Lie algebra $\L$ act on $\A$ as derivations.
This makes the space $\A$  an almost-graded module over $\L$.
In particular, we have
\begin{equation}
e_k\ldot A_m=mA_{k+m}+\sum_{h=k+m+1}^{k+m+M_3}\epsilon_{k,m}^h
A_h,\quad  \epsilon_{k,m}^h\in\C,
\end{equation}
with a constant $M_3$ not depending on $k$ and $m$.
All these constants $M_i$ can be easily given \cite{KNFa}. But
their exact value will not play any role in the following.

\bigskip

By point-wise multiplication, the space $\gb$ is a module over the 
associative algebra $\A$. Obviously the relations
\refE{glexp}, \refE{gldef}, \refE{sodef}, \refE{spdef}, 
are not disturbed. A direct calculation of the
possible orders at the points $P_+$ and $P_-$ shows that there
exists a constant $M_4$ (not depending on $k$ and $m$) such that
\begin{equation}
\A_k\cdot\gb_m\subseteq\bigoplus_{h=k+m}^{k+m+M_4}
\gb_h.
\end{equation}
In other words, $\gb$ is an almost-graded module over $\A$.
By considering the degree at $P_+$ we see that for $X\in\g$
\begin{equation}
A_m\cdot X_0=X_m+L,\quad L\in F_{m+1}.
\end{equation}
In general we do not have
$A_m\cdot X_0=X_m$ as the orders at $P_-$ will not coincide.
Also, as long as $\a\ne 0$ the element $A_m\cdot X$ is not
necessarily  an element of $\gb$, as $\a$ is not necessarily 
an eigenvector of $X$. 
Note that  $A_m\cdot X$ is always an element of the
Krichever-Novikov current algebra $\g\otimes\A$.
\bigskip
 
Next we introduce an  action of $\L$ on $\laxg$. 
Recall that $\gb=\glb(n)$ should be interpreted as the endomorphism algebra
of the space of meromorphic sections of a vector bundle. 
The action of $\L$ on $\gb$ should come from the action of $\L$ on 
these sections by taking the covariant derivative with respect to
some connection $\nabla^{(\omega)}$ with
a connection form $\omega$ \cite{Kiso}.

We introduce $\nabla^{(\omega)}$ following the lines of 
\cite{Klax}, \cite{Kiso} with certain modifications.
The connection form $\w$ should be  a $\g$-valued 
meromorphic 1-form, holomorphic outside $P_+$, $P_-$ and $W$, and
has a certain prescribed behavior at the points in $W$.
For $\gamma_s\in W$ with $\a_s= 0$ the requirement is that
$\w$ is also regular there.
For the points $\gamma_s$ with   $\a_s\ne 0$ we require that
it has  the  expansion
\begin{equation}\label{E:connl}
\w(z_s)=\left(\frac {\w_{s,-1}}{z_s}+\w_{s,0}+\w_{s,1}+
\sum_{k>1}\w_{s,k}z_s^k\right)dz_s.
\end{equation}
The following conditions were given in 
\cite{Klax} for $\glb(n)$ and for the other 
classical Lie algebras  in
\cite{KSlax}.
For $\gl(n)$ we take: there 
exist $\tb_s\in\C^n$ 
and $\tk_s\in \C$ such that 
\begin{equation}\label{E:gldefc}
\w_{s,-1}=\a_s \tb_s^{t},\quad
\w_{s,0}\,\a_s=\tk_s\a_s,
\quad
\tr(\w_{s,-1})=\tb_s^t \a_s=1.
\end{equation}
Note that compared to \refE{gldef} only the last condition was
modified.

For $\so(n)$ we take: 
there 
exist $\tb_s\in\C^n$ 
and $\tk_s\in \C$ such that
\begin{equation}\label{E:sodefc}
\w_{s,-1}=\a_s\tb_s^t-\tb_s\a_s^t,
\quad
\w_{s,0}\,\a_s=\tilde\ka_s\a_s,
\quad
\tb_s^t\a_s=1.
\end{equation}

For $\spn(2n)$ we take:
there exists 
$\tb_s\in\C^{2n}$, 
$\tilde\ka_s\in\C$ such that 
\begin{equation}\label{E:spdefc}
\w_{s,-1}=(\a_s\tb_s^t+\tb_s\a_s^t)\sigma,
\quad
\w_{s,0}\,\a_s=\tilde\kappa_s\a_s,
\quad 
\a^t_s\sigma\w_{s,1}\a_s=0,\quad
\tb_s^t\sigma\a_s=1.
\end{equation}
\begin{remark}
Compared to \refE{gldef}, \refE{sodef}, \refE{spdef} only the 
condition $\b_s^{t}\a^s=0$ (resp.  $\b_s^{t}\sigma\a^s=0$) was replaced by
$\tb_s^{t}\a^s=1$ (resp.  $\tb_s^{t}\sigma\a^s=1$). 
For $\spn(2n)$ we could also allow  additional poles of order two
at the points $\ga_s$ of the form
$(\tilde\nu\a_s\a^t_s\sigma)/z_s^2$ without changing anything in the
following.
\end{remark}

In the same way as in \cite{KSlax} the existence 
of the elements of $\gb_m$ is shown, one shows that there exist
many connections  fulfilling these conditions.
We might even require that
the connection form is holomorphic at $P_+$, and  we
will do this in the following without any further mentioning.
Note also that if all $\a_s=0$ we could take $\w=0$.

The induced  connection for the algebra will be
\begin{equation}\label{E:conng}
\nabla^{(\w)}=d+[\w,.].
\end{equation}
If $\w$ is fixed we will  usually drop it in the notation.
Let $e$ be a vector field. 
In a local coordinate $z$ the connection form and the vector field 
are represented as $\omega=\tilde\omega dz$ and 
 $e=\tilde e\frac{d}{dz}$ 
with a local function $\tilde e$ and a local
matrix valued function  $\tilde\omega$.
The covariant derivative in direction of $e$ is given by 
\begin{equation}\label{E:covder}
\nabla_e^{(\w)}=dz(e)\frac {d}{dz}+[\w(e),.\,]=
e\ldot +[\,\tilde\omega\tilde e\, ,.\,]
=\tilde e\cdot \big(\frac {d}{dz}+[\,\tilde\omega\, ,.\,]\big).
\end{equation}
Here the first term corresponds to taking the usual
derivative of functions in each matrix element separately.

Using the last description we can easily verify for 
$L\in\gb,\  g\in \A,\  e,f\in\L$
\begin{equation}\label{E:conn1r}
\nabla_e^{(\w)}(g\cdot L)=(e\ldot g)\cdot L +
g\cdot \nabla_e^{(\w)}L,
\qquad
\nabla_{g\cdot e}^{(\w)}L=g\cdot \nabla_e^{(\w)}L,
\end{equation}
and
\begin{equation}\label{E:conn2r}
\nabla_{[e,f]}^{(\w)}=[\nabla_{e}^{(\w)},\nabla_{f}^{(\w)}].
\end{equation}
\begin{proposition}\label{P:deri}
$\nabla_e^{(\w)}$  acts as a derivation
on the Lie algebra $\laxg$, i.e.
\begin{equation}\label{E:deri}
\nabla_e^{(\w)}[L,L']=
[\nabla_e^{(\w)}L,L']+
[L,\nabla_e^{(\w)}L'].
\end{equation}
\end{proposition}
\begin{proof}
First note that the local representing function $\tilde e$ commutes with all
the matrices. Then 
\begin{align*}
\nabla_e^{(\w)}[L,L']&=\tilde e\cdot(\frac{d[L,L']}{dz}+[\tilde\w,[L,L']])
\\
        &=\tilde e\cdot([\frac {dL}{dz},L']+ [L,\frac {dL'}{dz}]+
            [\tilde\w,[L,L']])      
\\
[\nabla_e^{(\w)}L,L']&=\tilde e\cdot([\frac {dL}{dz},L']+[[\tilde\w,L],L'])
\\
[L,\nabla_e^{(\w)}L']&=\tilde e\cdot([L,\frac {dL'}{dz}]+[L,[\tilde\w,L']]).
\end{align*}
Equation \refE{deri} follows from the Jacobi identity for the
matrix commutator.
\end{proof}

\begin{proposition}\label{P:action}
The covariant derivative makes 
$\laxg$  to  a Lie module over  $\L$.
\end{proposition}
\begin{proof}
As the connection form has values in $\g$, for $L\in\laxg$ the covariant 
derivative $\nabla^{(\w)}_e L$ will be a $\g$-valued meromorphic
function.  Clearly there will be no additional poles. We have to 
check that the behavior at the points of the weak singularities is
as prescribed. In particular, we have to check that there are no
poles of order two (of order $\ge 3$ for $\spnb(2n)$).
By \refE{conn2r} it follows that $\laxg$ will be a Lie module over
$\L$.
Here we will only consider the case $\gl(n)$ and postpone
$\so(n)$ and $\spn(2n)$ to the Appendix  \ref{S:action}.
\newline
Let $\ga_s$ be a point in $W$. 
If $\a_s=0$ then the Lax operators neither  have poles at 
$\ga_s$ nor fulfill any  condition on the
zero and first order expansions.
By requirement, our connection form is holomorphic at
$\ga_s$ and $\nabla^{(\w)}_e L$ has the correct
behavior at $\ga_s$.
Hence, the only non-trivial case to consider is 
$\a_s\ne 0$.
For simplicity  we will omit the index $s$.
In particular, $z$  will denote $z_s$.
As $\tilde e$ evaluated at $\ga_s$ is a scalar we might ignore it
in the calculation. Also we use the same symbol for $\w$ and its 
representing matrix function.
We take the expansions
obeying the conditions \refE{gldef} and \refE{gldefc} respectively:
\begin{equation}\label{E:expact}
L(z)=\frac {L_{-1}}{z}+L_0+L_1z+O(z^2),
\qquad
\w(z)=\frac {\w_{-1}}{z}+\w_0+\w_1z+O(z^2).
\end{equation}
Hence 
\begin{equation}
\frac {dL}{dz}(z)=\frac {-L_{-1}}{z^2}+L_1+O(z^1),
\end{equation}
and
\begin{multline}
[\w,L]= (1/z^2)[\w_{-1},L_{-1}]+
(1/z)\left([\w_{-1},L_{0}]+[\w_{0},L_{-1}]\right)+
\\
\left([\w_{-1},L_{1}]+[\w_{0},L_{0}]+[\w_{1},L_{-1}]\right).
\qquad 
\end{multline}
For the pole of order two we calculate the  (matrix) coefficient as
\begin{equation}
-L_{-1}+[\w_{-1},L_{-1}]=
-\a\b^t+[\a\tb^t,\a\b^t]=
-\a\b^t+\a\tb^t\a\b^t-
\a\b^t\a\tb^t=0.
\end{equation}
Here we used 
$\tb^t\a=1$ and $\b^t\a=0$.
\newline
The matrix coefficient of the pole of order one is
\begin{multline}
[\w_{-1},L_{0}]+[\w_{0},L_{-1}]=
\a\tb^tL_0-L_0\a\tb^t-\w_0\a\b^t+\a\b^t\w_0
\\
=
\a\left(\tb^tL_0-\ka\tb^t-\tk\b^t+\b^t\w_0\right)
=\a\hat\b^t,
\end{multline}
where we take the row vector defined by the second factor
as $\hat\b^t$.
We calculate
\begin{equation}
\hat\b^t\a=(\tb^tL_0-\ka\tb^t-\tk\b^t+\b^t\w_0)\a=
\ka\tb^t\a-\ka\tb^t\a-\tk\b^t\a+\tk\b^t\a=0.
\end{equation}
Here we used several times $L_0\a=\ka\a$ and
$\w_0\a=\tk\a$.
\newline
Finally we have to show that the zero degree term has 
$\a$ as an eigenvector, i.e. that
the vector 
\begin{equation}
L_1\a+[\w_{-1},L_{1}]\a
+[\w_{0},L_{0}]\a
+[\w_{1},L_{-1}]\a
\end{equation}
is a multiple of $\a$. First note that 
$[\w_{0},L_{0}]\a=0$ as $\a$ is an eigenvector for both matrices.
It remains
\begin{equation}
L_1\a+
\a\tb^t L_1\a-L_1\a\tb^t\a+\w_1\a\b^t\a-\a\b^t\w_1\a=
\a(\tb^tL_1\a-\b^t\w_1\a).
\end{equation}
Note that the second factor is  a scalar. Hence the
claim.
\newline
The proofs for $\so(n)$ and $\spn(2n)$ are similar in 
spirit, but the calculations are more involved.
\end{proof}

\begin{proposition}
The decomposition
$\glb(n)=\snb(n)\oplus\slnb(n)$ is a decomposition into
$\L$-modules, i.e.
\begin{equation}
{(\nabla_e)}_{\snb(n)}: \snb(n)\to \snb(n),\qquad
{(\nabla_e)}_{\slnb(n)}: \slnb(n)\to \slnb(n).
\end{equation}
Moreover,  via the 
identification \refE{sint} the $\L$-module $\snb(n)$ is 
equivalent to the  $\L$-module $\A$.
\end{proposition}
\begin{proof}
The Equation \refE{covder}, applied to the trace-less matrices, yields
a trace-less matrix as the commutator has trace zero. 
Hence we end up in $\slnb(n)$.
For the
scalar matrices the commutator even vanishes. Hence 
${(\nabla_e)}_{\snb(n)}$ does not depend on the connection form
and only the usual action of $\L$ on $\A$ is present.
\end{proof}

\begin{proposition}\label{P:almgrad}
$ $

\noindent
(a)  $\gb$  is an almost-graded $\L$-module.

\noindent
(b) At the lower bound we have 
\begin{equation}\label{E:allstr}
\nabla_{e_k}X_m=m\cdot X_{k+m}+L,
\quad  L\in F_{k+m+1}.
\end{equation}
\end{proposition}
\begin{proof}
(a) We write \refE{covder}  for homogenous elements 
\begin{equation}\label{E:covex}
\nabla_{e_k}X_m=e_k\ldot X_m+[\,\tilde\omega\tilde e_k\, ,X_m].
\end{equation}
The form $\w$ has fixed order at $P_+$ and $P_-$,
the action of $\L$ on $\A$ is almost-graded, and 
the bracket corresponds to the bracket in the almost-graded $\gb$.
Altogether 
this yields the claim.
\newline
(b) 
Locally at $P_+$
\begin{equation}
X_m=Xz_+^m+O(z_+^{m+1}),\qquad 
e_k=z_+^{k+1}\frac {d}{dz}+O(z_+^{k+2}).
\end{equation}
This implies
\begin{equation}
e_k\ldot X_m=m X z_+^{k+m}+O(z_+^{k+m+1}),
\qquad
\tilde\omega\tilde e_k=Bz_+^{k+1}+O(z_+^{k+2}),
\end{equation}
with $B\in\gl(n)$. Hence 
\begin{equation}
[\,\tilde\omega\tilde e_k\, ,X_m]=O(z^{k+m+1}),
\end{equation}
and \refE{allstr} follows from \refE{covex}.
\end{proof}

{If $\omega$ has a pole of order 1 at $P_+$ the lower bound
will still  be of degree $k+m$ but the coefficients will be
different.
}

\subsection{Module structure over $\D$ and the $\D_{\g}$ algebra}
\label{S:d1alg}
$ $

The Lie algebra  $\D$ of meromorphic differential operators on $\Sigma$ of 
degree $\le 1$ holomorphic outside of $\{P_+,P_-\}$ is defined
as the semi-direct sum of $\A$ and $\L$ given by the action of $\L$ on
$A$. As vector space $\D=\A\oplus\L$ with the Lie bracket
\begin{equation}
[(g,e),(h,f)]:=(e\ldot h-f\ldot g,[e,f]).
\end{equation}
In particular
\begin{equation}\label{E:relsd}
[e,h]=e\ldot h.
\end{equation}
It is an almost-graded Lie algebra \cite{Scocyc}.
\begin{proposition}
The Lax operator algebras $\laxg$ are almost-graded 
Lie modules over $\D$ via
\begin{equation}
e\ldot L:=\nabla_e^{(\w)}L,\qquad
h\ldot L:=h\cdot L.
\end{equation}
\end{proposition}
\begin{proof}
As $\laxg$ are almost-graded $\A$- and $\L$-modules it is 
enough to  show that the relation \refE{relsd} is satisfied.
For $e\in \L, h\in\A, L\in\laxg$ using \refE{covder} we get
\begin{multline*}
e\ldot (h\ldot L)-h\ldot(e\ldot L)=
\nabla_e^{(\w)}(h L)-h \nabla_e^{(\w)}(L)
=
\\
\tilde e\left(\frac {d(hL)}{dz}+[\tilde w, hL]\right)
-h\tilde e\left(\frac {dL}{dz}+[\tilde w, L]\right)
= \left(\tilde e\frac {dh}{dz}\right)L=
(e\ldot h)L=[e,h]\ldot L.
\end{multline*}
\end{proof}

In this context another structure shows up.
The Lax operator algebra $\gb$ is a Lie module over $\L$. \refP{deri} 
says that this action of $\L$ on $\laxg$ is an action by derivations.
Hence as above we can consider the semi-direct sum
$\D_{\g}=\laxg\oplus\L$ with Lie product given by
\begin{equation}\label{E:dac}
[e,L]:=e\ldot L=\nabla_e^{(\w)}L,
\end{equation}
for the mixed terms. See \cite{Saff} for the corresponding construction for the
classical Krichever-Novikov algebras of affine type.
Similar to \cite{Saff} also almost-graded central extensions of 
$\D_{\g}$ can be studied and classified.
Details will be given elsewhere.
\section{Cocycles}
\label{S:cocylces}
\subsection{Geometric cocycles}
$ $

In the following we introduce geometric (Lie algebra) 2-cocycles
of $\laxg$ with values in the trivial module $\C$. The corresponding
cohomology space $\H^2(\laxg,\C)$ classifies equivalence classes of
(one-)dimensional central extensions of $\laxg$.

Recall that a 2-cocycle for $\laxg$ is a bilinear form
$\ga:\laxg\times\laxg\to\C$ which is (1) antisymmetric
and (2) fulfills the cocycle condition
\begin{equation}\label{E:cohcoc}
\ga([L,L'],L'')+
\ga([L',L''],L)+
\ga([L'',L],L')=0.
\end{equation}
A 2-cocycle $\ga$ is a coboundary if there exists a linear form $\phi$ 
on $\laxg$ with 
\begin{equation}
\ga(L,L')=\phi([L,L']),\qquad  L,L'\in\laxg.
\end{equation}

The relation to  central extensions of $\laxg$ is as follows.
Given a 2-cocycle $\ga$ for $\laxg$, the associated central extension
$\gh_{\ga}$ is given as vector space direct sum $\gh_\ga=\gb\oplus\C\cdot t$
with Lie product given by
\begin{equation}\label{E:centextf}
[\widehat{L},\widehat{L'}]=\widehat{[L,L']}+\ga(L,L')\cdot t,
\quad [\widehat{L},t]=0,\qquad
L,L'\in\laxg.
\end{equation}
Here we used $\widehat{L}:=(L,0)$ and $t:=(0,1)$.
Vice versa, every central extension 
\begin{equation}
\begin{CD}
0@>>>\C@>i_2>>\gh@>p_1>>\gb@>>>0,
\end{CD}
\end{equation}
defines  a 2-cocycle
$\ga:\gb\to\C$ by choosing a section $s:\gb\to \gh$.

Two central extensions $\gh_{\ga}$ and $\gh_{\ga'}$ are equivalent
if the defining cocycles $\ga$ and $\ga'$ are cohomologous, i.e.
their difference is a coboundary.

\bigskip

Let $\w$ be a connection form as introduced in the last section
for defining the connection \refE{conng}.
Furthermore, let $C$ be a differentiable cycle on $\Sigma$ not meeting
$\{P_+,P_-\}\cup W$. We define the following cocycles for $\laxg$:
\begin{equation}\label{E:g1}
\ga_{1,\w,C}(L,L')= \cinc{C} \tr(L\cdot \nabla^{(\w)}L'),
\qquad L,L'\in\laxg, 
\end{equation}
and
\begin{equation}\label{E:g2}
\ga_{2,\w,C}(L,L')= \cinc{C} \tr(L)\cdot \tr(\nabla^{(\w)}L'),
\qquad L,L'\in\laxg.
\end{equation}
\begin{proposition}\label{P:c12cc}
The bilinear forms $\ga_{1,\w,C}$ and $\ga_{2,\w,C}$
are cocycles.
\end{proposition}
\begin{proof}
We start with $\ga_{2,\w,C}$. For the integration form
we calculate
$$
\tr(L)\cdot \tr( \nabla^{(\w)}L')=
\tr(L)\cdot \tr(dL'+[\w,L'])=
\tr(L)\cdot \tr(dL').
$$
Now $h:=\tr(L)\tr(L')$ is a meromorphic function and
$$
dh=d(\tr(L)\tr(L'))=(d(\tr(L)))\tr(L')+\tr(L)d(\tr(L')).
$$
By Stokes' theorem $\cinc{C}dh=0$ and hence 
$$
\cinc{C}\tr(L)\tr(dL')=
-\cinc{C}\tr(L')\tr(dL), 
$$
which is the antisymmetry.
Obviously, $\ga_{2,\w,C}([L,L'],L'')=0$ and the condition
\refE{cohcoc} is true.
\newline
Next we consider  $\ga_{1,\w,C}$ and write $\w=\tilde \w dz$ in
local coordinates.
The integration form can be written as
\begin{equation}\label{E:cocycum1}
\tr(L\cdot \nabla^{(\w)}L')=
\tr(L\cdot (dL'+[\tilde \w,L']dz))=
\tr(L\cdot dL')+
\tr(L\cdot [\tilde \w,L'])dz.
\end{equation}
Set $h:=\tr(L\cdot L')$ then
$$
dh=d(\tr(L\cdot L'))=
\tr(dL\cdot L')+\tr(L\cdot dL')
=\tr(L'\cdot dL)+\tr(L\cdot dL').
$$
By Stokes' theorem the integral over $dh$ vanishes, hence the
first term 
in \refE{cocycum1}
is anti-symmetric.
For the second term we calculate
\begin{equation}
\tr(L\cdot [\tilde \w,L'])=
\tr(L\cdot\tilde\w\cdot L'-L\cdot L'\cdot\tilde\w)=
\tr(L'\cdot L\cdot \tilde\w -L'\cdot \tilde\w\cdot L)=
-\tr (L'\cdot [\tilde \w,L]).
\end{equation}
Hence also the second term is antisymmetric.
\newline
For the cocycle condition we consider
\begin{equation}\label{E:39}
\tr([L,L']\cdot \nabla^{(\w)}L'')=
\tr([L,L']\cdot dL'')+
\tr([L,L']\cdot [\w,L'']).
\end{equation}
First we consider the 2nd summand. It calculates 
(using the trace property) as
\begin{equation}
\tr([L'',[L,L']]\cdot \w).
\end{equation}
Cyclically permuting $L,L',L''$ and summing up the results 
gives zero by the Jacobi identity.
For the first summand in \refE{39}  we get
\begin{equation}
\tr(L\cdot L'\cdot dL''- L'\cdot L\cdot dL'').
\end{equation}
Cyclically permuting $L,L',L''$ and summing up the result we obtain
(again using the trace property) the exact form
\begin{equation}
d\big(\tr(L\cdot L'\cdot L'')-\tr(L\cdot L''\cdot L')\big).
\end{equation}
Hence integration over a closed cycle is equal to  zero and
the cocycle condition is shown.
\end{proof}
\begin{proposition}\label{P:cind}
$ $

\noindent
(a) The cocycle  $\ga_{2,\w,C}$ does not depend on the choice of the
connection form $\w$.

\noindent
(b) The cohomology class  $[\ga_{1,\w,C}]$ does not depend on
the choice of the
connection form $\w$.
More precisely 
\begin{equation}\label{E:cobw}
\gamma_{1,\w,C}(L,L')-\gamma_{1,\w',C}(L,L')=
\cinc{C}\tr\big((\w-   \w')[L,L']\big)
\end{equation}
\end{proposition}
\begin{proof}
As it follows from the proof of the last proposition
$\ga_{2,\w,C}$ is indeed independent of $\w$.
\newline
Let $\w$ and $\w'$ be two connection forms and set $\theta=\w-\w'$.
Then 
\begin{multline}
\gamma_{1,\w,C}(L,L')-\gamma_{1,\w',C}(L,L')=
\cinc{C}\tr(L\cdot (\nabla^{(\w)}-\nabla^{(\w')})L')=
\\
\cinc{C}\tr(L\cdot [\theta,L'])dz.\qquad
\end{multline}
{}From the trace property we get
\begin{equation}\label{E:cocycum2}
\tr(L\cdot [\theta,L'])=
\tr(L\cdot \theta \cdot L'-L\cdot L'\cdot \theta)=-
\tr(\theta\cdot(L\cdot L'-L'\cdot L))=-
\tr(\theta\cdot [L,L']).
\end{equation}
If we define the linear form 
\begin{equation}
\psi_{\theta,C}(L):=\cinc{C}\tr(\theta\cdot L)
\end{equation}
on $\laxg$ we see that 
\begin{equation}
\gamma_{1,\w,C}(L,L')-\gamma_{1,\w',C}(L,L')=\psi_{-\theta,C}([L,L']).
\end{equation}
Hence the difference is a coboundary as claimed.
\end{proof}
\begin{remark}
Using \refE{cocycum1} and \refE{cocycum2} the cocycle $\ga_{1,\w,C}$ can be 
rewritten as
\begin{equation}\label{E:coeq}
\ga_{1,\w,C}(L,L')=
\cinc{C}\tr\big(LdL'-\w\cdot [L,L']).
\end{equation}
This is the form of the cocycle
(for $C$ a circle around $P_+$) defined and studied in \cite{KSlax}.
\end{remark}

As $\ga_{2,\w,C}$ does not depend on $\w$ we will drop $\w$  in the
notation. 
Note that $\gamma_{2,C}$ vanishes on $\laxg$ for $\g=\sln(n),\so(n),\spn(2n)$.
But it does not vanish on $\laxs(n)$, hence not on $\laxgl(n)$.

\subsection{$\L$-invariant cocycles}
$ $

Recall that after fixing a connection form $\w'$ the vector field algebra
$\L$ operates via the covariant derivative $e\mapsto \nabla^{(\w')}_e$
on $\laxg$, see \refE{conng}. Later we will assume that $\w=\w'$.

\begin{definition}
A cocycle for $\laxg$ is called {\it $\L$-invariant} 
(with respect to $\w'$)  if
\begin{equation}
\gamma(\nabla^{(\w')}_{e}L,L')+
\gamma(L,\nabla^{(\w')}_{e}L')=0,\qquad
\forall e\in\L,\quad \forall L,L'\in \laxg.
\end{equation}
\end{definition}

\begin{proposition}\label{P:linv}

(a) The cocycle $\gamma_{2,C}$ is $\L$-invariant.

(b) If $\w=\w'$ then the cocycle $\gamma_{1,\w,C}$ is $\L$-invariant.
\end{proposition}
\begin{proof}
As the cocycles are antisymmetric the $\L$-invariance can be written
as
\begin{equation}
\gamma(\nabla^{(\w')}_{e}L,L')=
\gamma(\nabla^{(\w')}_{e}L',L),\quad \forall e\in\L,\quad 
\forall L,L'\in\laxg.
\end{equation}
In the following we will write locally $e=\tilde e\frac {d}{dz}$,
$\w=\tilde\w dz$ and $\w'=\tilde\w 'dz$.
\newline
First we consider $\ga_{2,C}$.
For the integration form we calculate
\begin{equation}
\tr(\nabla^{(\w')}_{e}L)\cdot \tr(\nabla^{(\w)}L')
=\tr(e\ldot L)\cdot \tr(\frac {dL'}{dz}dz)=
\tilde e \cdot dz \cdot \tr(\frac {dL}{dz})\cdot  \tr(\frac {dL'}{dz}).
\end{equation}
Permuting  $L$ and $L'$ does not change the
expression. Hence $\ga_{2,C}$ is $\L$-invariant.
\newline
Next we consider $\ga_{1,\w,C}$.
For the integration form we obtain
\begin{multline}
\tr(\nabla^{(\w')}_{e}L\cdot \nabla^{(\w)}L')=
\tr\big((\frac {dL}{dz}\tilde e+[\tilde\w'\cdot e,L])
(\frac {dL'}{dz}+[\tilde\w,L'])\big)
=
\\
\tilde e\cdot dz\cdot\tr\big(
(\frac {dL}{dz}+[\tilde\w',L])
(\frac {dL'}{dz}+[\tilde\w,L'])\big).\qquad \qquad
\end{multline}
Since  $\w=\w'$  after applying the trace this expression is obviously
invariant if we interchange $L$ and $L'$. Hence the claim.
\end{proof}
\noindent
In the case that $\g$ is simple and the integration cycle $C$ is a
separating cycle (see \refS{local}) then
in statement (b) we even have ``if and only  if'', see  \refP{boundlinv}.

\medskip
We call a cohomology class {\it $\L$-invariant} if it has a representing
cocycle which is  $\L$-invariant.
The reader should be warned that this does not mean that all
representing cocycles are  $\L$-invariant. 
On the contrary, in \refT{main} we will show that up
to a scalar multiple there is at most one 
 $\L$-invariant representing 
cocycle.
Clearly, the  $\L$-invariant classes constitute a subspace of 
 $\H^2(\laxg,\C)$ which we denote by 
 $\H^2_{\L}(\laxg,\C)$.
\subsection{Some remarks on $\D_\g$ cocycles}\label{S:remarks}
$ $

For the following let $\w=\w'$. In this article mainly the
property of $\L$-invariance of a cocycle gives us
a very elegant way to single out a unique  element in
a cohomology class. But there is even 
a deeper meaning behind the definition.
In \refS{d1alg} we introduced the algebra $\D_\g$. The Lax operator 
algebra is a subalgebra of $\D_\g$.
Given a 2-cocycle $\ga$ for $\laxg$ we might extend it as a bilinear form
on  $\D_\g$ by setting
($L,L'\in\laxg$, $e,f\in\L$)
\begin{equation}
\tilde\ga(L,L')=\ga(L,L'),\quad
\tilde\ga(e,L)=\tilde\ga(L,e)=0,\quad
\tilde\ga(e,f)=0.
\end{equation}
\begin{proposition}
 The extension $\tilde\ga$ is a cocycle for  $\D_\g$ if and only
if $\ga$ is $\L$-invariant.
\end{proposition}
\begin{proof}
If we check the cocycle conditions on  $\tilde\ga$ (with respect to $\D_\g$) 
for elements of ``pure types'', i.e. elements which are either currents or vector fields,
we see that the only condition which is not automatic 
is of the type
\begin{equation}\label{E:cod}
\tilde\ga([L,L'],e)+
\tilde\ga([L',e],L)+
\tilde\ga([e,L],L')=0.
\end{equation}
Using \refE{dac} we get that \refE{cod} is true if an only if
\begin{equation}
\ga(\nabla_e^{(\w)}L,L')+
\ga(L,\nabla_e^{(\w)}L')=0.
\end{equation}
Hence, 
the claim.
\end{proof}
In \cite{Saff} it was shown that for the 
Krichever-Novikov current algebras the inverse is also true
in the following sense:
{\it Every local cocycle (see the definition below) for $\D_\g$
is cohomologous to a local cocycle which having been restricted to
$\gb$ is $\L$-invariant.}
In this way cocycles coming from 
projective representations of $\gb$ which admit an extension to
a projective representation of  $\D_\g$ yield
$\L$-invariant cocycles up to coboundaries.

Similar statements are true for the  $\D_\g$ associated to 
the Lax operator algebras $\laxg$. Details will appear elsewhere.

\subsection{Local Cocycles}\label{S:local}
$ $

A cocycle $\ga$ of the almost-graded Lie algebra $\laxg$ is called
{\it local} if there exist $R,S\in\Z$ such that
(see \cite{KNFa})
\begin{equation}\label{E:cloc}
\ga(\laxg_n,\laxg_m)\ne 0 \implies 
R\le n+m\le S.
\end{equation}
Local cocycles are important since exactly  in this case the
almost-grading of $\laxg$ can be extended to the central
extension $\gh_\ga$ 
\refE{centextf} by assigning the central element $t$ a certain
degree (e.g. the degree 0).

We call a cohomology class 
a {\it local cohomology class} if it admits a local representing cocycle. 
Again, not every representing cocycle of a local class
is local. Obviously, the set of local cohomology classes is a subspace
of  $\H^2(\laxg,\C)$ which we denote by 
 $\H^2_{loc}(\laxg,\C)$.
This space classifies up to equivalence 
central extensions of $\laxg$ which are 
almost-graded.
The cohomology classes admitting a local and $\L$-invariant
representing cocycle  constitute a subspace of 
$\H^2_{loc}(\laxg,\C)$ which we denote by
 $\H^2_{loc,\L}(\laxg,\C)$.

\medskip

For a general integration cycle $C$ the cocycles 
$\ga_{2,C}$ and $\ga_{1,\w,C}$ neither are  local  nor
define a local cohomology class.
But if we choose  a cycle $C_s$ separating $P_+$ from
$P_-$ as integration path  then, we will show, they are local. Such $C_s$
are homologous to 
circles around   $P_+$ with respect to the 
integration of differential forms without residues at points
different from   $P_\pm$. Hence for them the integration 
can be given by
evaluating the residue of the form at $P_+$, respectively at $P_-$.
In case that we integrate along a circle around  $P_+$  we will 
drop it in the notation of $\ga$.

\begin{proposition}\label{P:locg2}
The integration form $\tr(L)\cdot \tr(dL')$ does not have any poles 
besides  possibly at $P_\pm$. Furthermore,
\begin{equation}\label{E:ga2}
\ga_2(L,L')=\res_{P_+}(\tr(L)\cdot \tr(dL'))
\end{equation}
is a local  $\L$-invariant cocycle.
\end{proposition}
\begin{proof}
As already shown above the cocycle 
$\ga_2$ can be written as \refE{ga2}.
The matrices of $\slnb(n),\sob(n)$ and $\spnb(2n)$ are traceless, hence 
for them the cocycle vanishes. 
It remains to consider $\glb(n)$.
Set $h:=\tr(L)\cdot \tr(dL')$, which  is a meromorphic differential.
The order of $h$ at $P_+$ is bounded from below by 
$\ord_{P_+}(L)+\ord_{P_+}(L')-1$. By the definition of the homogenous
elements \refE{almdeg} $h$  will not have any  residue at $P_+$ if 
$\deg L+\deg L'>0$.
\newline
Following the prescription \refE{glexp} we get at the points $P_s\in W$
\begin{equation}
dL'=\frac {-L'_{s,-1}}{z_s^2}+L'_{s,1}+\sum_{k>1}L'_{s,k}kz^{k-1}.
\end{equation}
By Condition \refE{gldef} $\tr\, L_{s,-1} =\tr\, L_{s,-1}'=0$.
Hence neither $\tr \,L $ nor $\tr\,dL'$  have any poles at the weak
singularities $W$. This implies that the residue of $h$ at $P_+$ 
is the negative of its  
residue  at  $P_-$.
Using  \refE{almdeg} and considering the orders at $P_-$ we
see that there is a constant  $S$  such that if $\deg L+\deg L'<S$
the differential $h$ will not have
any pole there.
This shows locality.
The $\L$-invariance is \refP{linv}.
\end{proof}
\begin{proposition}[\cite{KSlax}]\label{P:locg1}
The integration form $\tr(L\cdot \nabla^{(\w)} L')$ does not have any poles 
other than possibly at $P_\pm$. Furthermore,
\begin{equation}\label{E:ga1}
\ga_{1,\w}(L,L')=\res_{P_+}(\tr(L\cdot \nabla^{(\w)} L'))
\end{equation}
is a local cocycle.
It will be $\L$-invariant if $\w$ coincides with the connection
form $\w'$ associated to the $\L$-action.
\end{proposition}
\begin{proof}
As noticed  above, the cocycle $\ga_{1,\w}$ can be written in
the form  \refE{coeq}. Hence \refE{ga1} is exactly 
the cocycle discussed in \cite{KSlax}. 
Its locality is stated there in 
Theorems  4.3, 4.6 and 4.9.
The $\L$-invariance follows from \refP{linv}.
\end{proof}

\subsection{Main theorem}
$ $

\begin{theorem}\label{T:main}
$ $

(a) If $\g$ is simple (i.e. $\gb=\sln(n), \so(n), \spn(2n)$) then 
the space of local cohomology  classes is one-dimensional.
If we fix any connection form $\w$ then the space 
will be generated by the class of $\ga_{1,\omega}$.
Every $\L$-invariant (with respect to the connection 
$\w$) local cocycle is a scalar multiple of 
$\gamma_{1,\omega}$.

(b) For $\laxg=\bgl$ the space of local 
cohomology classes which are $\L$-invariant
having been restricted to the scalar subalgebra is two-dimensional.
If we fix any connection form $\w$ then the space
will  be generated by the classes of the 
cocycles $\gamma_{1,\omega}$ and $\gamma_2$.
Every $\L$-invariant local cocycle is a linear combination of  
$\gamma_{1,\omega}$ and $\gamma_2$.
\end{theorem}
\begin{proof}
The technicalities of the proof will be covered in \refS{induction} and
\refS{direct}.
In particular, by \refP{lsimp} and \refP{lcom} it follows
that $\L$-invariant and local cocycles are necessarily linear
combinations of the claimed form.
Hence the theorem will follow  for the cohomology space
$\H_{loc,\L}(\gb,\C)$. For the abelian part we had to put the
$\L$-invariance into the requirements. Hence for this part
we are done. For the simple algebras, resp. the simple part, we 
have to show that in each local cohomology class there is an
$\L$-invariant representative. But by \refT{locuni} the space
$\H_{loc}(\gb,\C)$ is at most one-dimensional. As 
by \refP{nonbound} 
the local cocycle $\ga_{1,\omega}$  is not 
a coboundary, this space is
exactly one-dimensional and $\ga_{1,\omega}$ is 
its representing element.
\end{proof} 
\begin{corollary}
Let $\g$ be a simple classical Lie algebra and $\gb$ the associated
Lax operator  algebra. 
Let $\w$ be a fixed connection form.
Then in each
$[\ga]\in \H_{loc}(\gb,\C)$ there exists a unique 
representative $\ga'$ which is local and $\L$-invariant (with respect to
$\w$). 
Moreover, $\ga'=\a\ga_{1,\omega}$, with $\a\in\C$.
\end{corollary}

\begin{proposition}\label{P:nonbound}
The cocycle $\ga=\ga_{1,\omega}$ is not 
a coboundary.
\end{proposition}
\begin{proof}
Assume that $\ga$ is a coboundary. This means that there exists 
a linear form $\phi:\gb\to\C$ such that
\begin{equation}\label{E:contra}
\ga(L,L')=\res_{P_+}\tr(L\cdot\nabla L')=
\phi([L,L']).
\end{equation}
Take $H\in\fh$ with $\kappa(H,H)\ne 0$, 
where $\fh$ is the Cartan subalgebra of the simple part of 
$\g$ and $\kappa$  its Cartan-Killing form.
Furthermore, let $H_0\in\gb$  be the element fixed by
\refE{locex}.
In particular, we have
$H_0=H+O(z_+)$.
We set%
\footnote{Notice  that  $H_{(n)}$ and $H_{n}$, in general, 
are different but coincide up to 
higher order.}  
$H_{(n)}:=H_0\cdot A_n\in\gb$ 
and hence
$H_{(n)}=H\cdot A_n+O(z_+^{n+1})$.
In  the following, let $n\ne 0$.
We have
\begin{equation}
\nabla H_{(n)}=\nabla(H_0\cdot A_n)=
\nabla(H_0)\cdot A_n+H_0\;dA_n.
\end{equation}
The expression $\nabla H_0$ is of nonnegative order, $A_n$ is of
order $n$, $H_0$ of order 0 and $dA_n$ of order $n-1$ at the
point $P_+$. Hence
\begin{equation}
\nabla H_{(n)}=H_0\;dA_n+O(z_+^{n})dz_+.
\end{equation}
Now we calculate
\begin{equation}
\ga(H_{(-1)},H_{(1)})=
\res_{P_+}\tr(H_{(-1)}\cdot \nabla H_{(1)})=
\res_{P_+}\tr(H_0 A_{-1} H_0 dA_{1})=
\res_{P_+}\tr(H_0^2 \frac{dz_+}{z_+}).
\end{equation}
As $H^2_0=H^2+O(z_+)$ 
we obtain 
\begin{equation}\label{E:con1}
\ga(H_{(-1)},H_{(1)})=
\res_{P_+}(\tr (H^2)\frac{dz_+}{z_+})
=\tr (H^2)=\alpha\cdot\kappa(H,H)\ne 0,
\end{equation}
with a non-vanishing constant $\alpha$ relating the trace form
with the Cartan-Killing form.
But
\begin{equation}\label{E:con2}
[H_{(-1)},H_{(1)}]=
[H_0A_{-1}, H_0A_1]=
[H_0,H_0]A_{-1}A_{1}=0.
\end{equation}
The relations \refE{con1} and \refE{con2} 
are in contradiction to \refE{contra}.
\end{proof}
\medskip
\begin{proposition}\label{P:boundlinv}
$ $

\noindent
(a) Let $\ga$ be a local and $\L$-invariant cocycle which is a 
coboundary, then $\ga= 0$.

\noindent
(b) Let $\g$ be simple, then the cocycle $\ga_{1,\w'}$ is $\L$-invariant
with respect to $\w$, if and only if $\w=\w'$.
\end{proposition}
\begin{proof}
(a) By \refT{main}  we get $\ga=\a\ga_{1,\w}+\b\ga_2$, with $\b=0$ for
the case $\g$ is simple.
But none of these cocycles is a  coboundary. Hence $\a=\b=0$.

\noindent
(b) As $\ga_{1,\w}$ and  $\ga_{1,\w'}$
are local and $\L$-invariant with respect to $\w$ their difference
$\ga_{1,\w}-\ga_{1,\w'}$ is also local and $\L$-invariant.
By \refP{cind} it is  a coboundary. Hence by part (a)  
$\ga_{1,\w}-\ga_{1,\w'}= 0$.
Equation \refE{cobw} gives the explicit expression. Assume $\w\ne\w'$. 
Let $m$ be the order of 
\begin{equation}
\theta=\w-\w'=(\theta_mz_+^m+O(z_+^m))dz_+
\end{equation}
at the point $P_+$.
As  $\g$ is simple the trace form $\tr(A\cdot B)$ 
is  nondegenerate and  we find 
\begin{equation}
\hat\theta=\hat\theta_{-m-1}z^{-m-1}+O(z_+^{-m}),
\end{equation}
such that $\b=\tr(\theta_m\cdot\hat\theta_{-m-1})\ne 0$. By 
\refL{weak} we get 
$\hat\theta=[L,L']+L''$  with 
$\ord(L'')\ge -m$. Hence,
\begin{multline}
0\ne \b=\tr(\theta_m\cdot\hat\theta_{-m-1})=
\frac{1}{2\pi\i}\int_{C_s}\tr\left((\w-\w')\cdot( [L,L']+L'')\right)
\\
=\frac{1}{2\pi\i}\int_{C_s}\tr\left((\w-\w')\cdot [L,L']\right)
=\ga_{1,\w}(L,L')-\ga_{1,\w'}(L,L')=0
\end{multline}
which is a contradiction.
\end{proof}
\section{Uniqueness of $\L$-invariant cocycles}
\label{S:induction}

\subsection{General induction}
$ $

Recall that we have the decomposition $\laxg=\oplus_{n\in\Z}\laxg_n$ into
subspaces of homogenous elements of  degree $n$.
The subspace $\laxg_n$ is generated by the basis
$\{L_n^r\mid r=1,\ldots,\dim\g\}$.

In the following, let $\ga$ be an $\L$-invariant cocycle for
the algebra $\laxg$. We  only assume that it is bounded
from above, i.e. there exists a $K$ (independent of 
$n$ and $m$) such that
$\ga(\gb_n,\gb_m)\ne 0$ implies $n+m\le K$.
Furthermore, we recall that our connection $\omega$ needed
to define the action of $\L$ on $\laxg$ is chosen to be holomorphic 
at the point $P_+$.

For a pair $(L^r_n,L^s_m)$ of homogenous elements  we call
$n+m$ the {\it level} of the pair.
Following the strategy developed in \cite{Scocyc}
we will consider the cocycle values $\ga(L^r_n,L^s_m)$ of pairs of level 
$l=n+m$ and will make induction over the level.
By the boundedness from above, the cocycle values will vanish at 
all pairs
of sufficiently high level, and it will turn out that everything will 
be fixed by the values of the cocycle at level zero.
Finally, we will show uniqueness of the cocycle up to rescaling at
level zero.

For a cocycle $\ga$ evaluated for pairs of elements of level
$l$ we will use the symbol $\equiv$ to denote that the expressions are
the same on both sides of an equation up to values of $\ga$  at
higher level. 
This has to be understood in the following strong sense:
\begin{equation}
\sum \alpha^{n}_{r,s}\ga(L_n^r,L_{l-n}^s)\equiv 0,\qquad 
 \alpha^{n}_{r,s}\in\C
\end{equation}
means a congruence modulo a linear combination of values of $\ga$ at 
pairs of basis elements of level $l'>l$. The coefficients of that
linear combination, as well as the  $\alpha^{n}_{r,s}$, depend only  
on the structure of the  Lie algebra $\gb$ and do not depend on $\ga$.

We will also use the same symbol $\equiv$ for equalities in  $\gb$ which
are true modulo terms of higher degree compared to the terms
explicitly written down.

\medskip

By the $\L$-invariance we have 
\begin{equation}
\ga(\nabla_{e_p}L_m^r,L_n^s)+
\ga(L_m^r,\nabla_{e_p}L_n^s)=0.
\end{equation}
Using the almost-graded structure \refE{allstr}
we obtain the following useful formula 
\begin{equation}\label{E:recform}
m\ga(L_{p+m}^r,L_n^{s})+
n\ga(L_{m}^r,L_{n+p}^{s})\equiv 0,
\end{equation}
valid for all $n,m,p\in\Z$.

\begin{proposition}\label{P:ln0}
Let $m+n\ne 0$ then at  level $m+n$ we have
\begin{equation}\label{E:ln0}
\ga(\laxg_m,\laxg_n)\equiv 0.
\end{equation}
\end{proposition}
\begin{proof}
In \refE{recform} we set $p=0$ and obtain
\begin{equation}\label{E:recform0}
(m+n)\ga(L_{m}^r,L_n^{s})
\equiv 0,
\end{equation}
Hence for $m+n\ne 0$ it follows
that $\ga(L_{m}^r,L_n^{s})\equiv 0$.
\end{proof}
\begin{proposition}\label{P:zerodeg}
\begin{equation}\label{E:zerodeg}
\ga(L^r_m,L_0^s)\equiv 0,\qquad
\forall m\in\Z.
\end{equation}
\end{proposition}
\begin{proof}
We  evaluate \refE{recform}
 for the values $m=1$ and $n=0$ and obtain the result.
\end{proof}
\begin{proposition}\label{P:zerobound}

\noindent
(a) We have $\ga(\laxg_n,\laxg_m)=0$ if $n+m>0$, i.e. 
the cocycle is bounded from above by zero.

\noindent
(b) If  $\ga(\laxg_n,\laxg_{-n})=0$ then the cocycle $\ga$ vanishes
identically.
\end{proposition}
\begin{proof}
If $\ga=0$ there is nothing to show. Hence assume $\ga\ne 0$.
As $\ga$ is bounded from above, there will be a smallest upper bound $l$,
such that above $l$ all cocycle values will vanish.
Assume that $l>0$ then by  \refP{ln0} the values at level $l$ are 
expressions of levels bigger than $l$. But there the
cocycle values  vanish. Hence also at level $l$. This is a contradiction
which shows (a). 
\newline
By induction using again  \refP{ln0} it follows that if everything
vanishes in level 0, the cocycle itself will vanish.
Hence, (b).
\end{proof}
Combining Propositions \ref{P:zerodeg} and \ref{P:zerobound} we obtain
\begin{corollary}\label{C:zerodeg}
\begin{equation}\label{E:zerodg}
\ga(L^r_m,L_0^s)= 0,\qquad
\forall  m\ge 0.
\end{equation}
\end{corollary}
\begin{proposition}
\begin{equation}\label{E:nm1}
\ga(L_{n}^r,L_{-n}^s)=n\cdot \ga(L_{1}^r,L_{-1}^s),
\end{equation}
\begin{equation}\label{E:p1m1}
\ga(L_{1}^r,L_{-1}^s)=\ga(L_{1}^s,L_{-1}^r).
\end{equation}
\end{proposition}
\begin{proof}
In \refE{recform} we take
the values $n=-k$, $m=1$ and $p=k-1$. This yields the expression 
\refE{nm1} up
to higher level terms. But as the level is zero, the higher 
level terms vanish.
Setting $n=-1$ we obtain \refE{p1m1}.
\end{proof}

\medskip
Before we go on let us summarize the results 
obtained up to now. Independently of the
structure of the Lie algebra $\g$, we obtain the following results
for every $\L$-invariant and bounded cocycle $\ga$:
\begin{enumerate}
\item
The cocycle is bounded from above by zero.
\item
The cocycle is uniquely given by its values at level zero.
\item
At level zero the cocycle is uniquely fixed by its values
$\ga(L_{1}^r,L_{-1}^s)$, for $r,s=1,\ldots,\dim\g$.
\item
The other cocycle values at level zero are given by
$\ga(L_{0}^r,L_{0}^s)=0$ and  $\ga(L_{n}^r,L_{-n}^s)$
given by \refE{nm1}.
\end{enumerate}

Let $X\in\g$ then we denote by 
$\tX_n$ any element
in $\laxg$ with leading term $Xz_+^n$ 
at $P_+$.
We define 
\begin{equation}\label{E:cocar}
\psi:\g\times\g\to\C\qquad
\psi_{\ga}(X,Y):=\ga(\tX_1,\tY_{-1}).
\end{equation}
As the cocycle vanishes for level greater zero, $\psi$ does not depend
on the choice of  $\tX_1$ and $\tY_{-1}$.
Obviously, it is a bilinear form on $\g$.
\begin{proposition}
(a) $\psi_{\ga}$ is symmetric, i.e.
$\psi_{\ga}(X,Y)=\psi_{\ga}(Y,X)$.
\newline
(b)  $\psi_{\ga}$ is invariant, i.e.
\begin{equation}
\psi_{\ga}([X,Y],Z)=\psi_{\ga}(X,[Y,Z]).
\end{equation}
\end{proposition}
\begin{proof}
First we have by \refE{p1m1} 
\begin{equation*}
\psi_{\ga}(X,Y)=\ga(\tX_1,\tY_{-1})= 
\ga(\tY_1,\tX_{-1})
=\psi_{\ga}(Y,X).
\end{equation*}
This is the symmetry.
Furthermore, using $[\tX_1,\tY_0]\equiv\widetilde{[X,Y]}_1$,
the fact that the cocycle vanishes for positive level, and by
the cocycle condition
\begin{multline*}
\psi_{\ga}([X,Y],Z)=\ga(\widetilde{[X,Y]}_1,\tZ_{-1})=
\ga([\tX_1,\tY_0],\tZ_{-1})=
\\
-\ga([\tY_0,\tZ_{-1}],\tX_1)
-\ga([\tZ_{-1},\tX_{1}],\tY_0).
\end{multline*}
The last term vanishes due to \refC{zerodeg}. Hence
\begin{equation*}
\psi_{\ga}([X,Y],Z)=\ga(\tX_1,[\tY_0,\tZ_{-1}])=
\ga(\tX_1,\widetilde{[Y,Z]}_{-1})=
\psi_{\ga}(X,[Y,Z]).
\end{equation*}
\end{proof}

As the cocycle $\ga$ is fixed by the 
values $\ga(L_{1}^r,L_{-1}^s)$, and they are fixed by the bilinear
map $\psi_{\ga}$ we proved:
\begin{theorem}\label{T:fixing}
Let $\gamma$ be an $\L$-invariant cocycle for $\laxg$ which is bounded 
from above by zero. Then $\ga$ is completely fixed by 
the  associated symmetric and invariant
bilinear form $\psi_\ga$ on $\g$ defined via \refE{cocar}.
\end{theorem}

\subsection{The case of simple Lie algebras $\g$}
$ $

By \refT{fixing} the cocycle is fixed by the associated
$\psi_{\ga}$ which is symmetric and $\L$-invariant.
For a finite-dimensional simple Lie algebra every such
form is a multiple of the Cartan-Killing form $\kappa$.
This supplies the proof of the
uniqueness of the cocycle. The existence is clear as 
$\gamma_{1,\omega}$, see \refE{ga1}, is an $\L$-invariant and local 
cocycle.
Hence, we obtain that every local and $\L$-invariant cocycle is
a scalar multiple of $\gamma_{1,\omega}$. 
By \refP{nonbound}, $\gamma_{1,\omega}$ is not  a coboundary. We obtain
\begin{proposition}\label{P:lsimp}
Let $\g$ be simple, then
\begin{equation}
\dim\H_{loc,\L}(\gb,\C)=1,
\end{equation}
and this cohomology space is generated by the class of $\gamma_{1,\omega}$.
Moreover, every $\L$-invariant cocycle which is bounded from above is local.
\end{proposition}

 \subsection{The case of $\laxgl(n)$}
$ $ 

First note that we have the direct decomposition, as Lie algebras,
$\bgl=\laxs(n)\oplus\laxsl(n)$. 
Let $\ga$ be a cocycle of $\laxgl(n)$ and denote by
$\ga'$ and $\ga''$ its restriction to $\laxs(n)$ and  $\laxsl(n)$
respectively. 
\begin{proposition}
\begin{equation}
\ga(x,y)=0,\quad \forall x\in \laxs(n),\ y\in \laxsl(n).
\end{equation}
\end{proposition}
\begin{proof}
Let $M$ be an upper bound for the cocycle $\ga$. Take $x$ and $y$ 
as above. In particular there is an $m$ such that 
$x$ can be written as linear combinations of basis elements of
degree $\ge m$.
By \refL{weak} there exist elements $y^{(i)}_1,y^{(i)}_2\in \laxsl(n)$,
$i=1,\ldots,k$, and $B\in \laxsl(n)$ with $B$ a linear combination
of elements of degree $\ge M-m+1$
such that $y=\sum_{i=1}^k[y^{(i)}_1,y^{(i)}_2]+B$.
Now
\begin{equation}
\ga(x,y)=\ga(x,\sum_{i=1}^k[y^{(i)}_1,y^{(i)}_2]+B)=
\sum_{i=1}^k\ga(x,[y^{(i)}_1,y^{(i)}_2])+
 \ga(x,B).
\end{equation} 
The last summand vanishes as the cocycle is bounded by $M$.
For the rest we calculate using the cocycle conditions
\begin{equation}
\ga(x,[y_1^{(i)},y_2^{(i)}])=\ga([x,y_1^{(i)}],y_2^{(i)})+
\ga([x,y_1^{(i)}],y_2^{(i)}).
\end{equation}
The commutators inside vanish 
since $sbn(n)$ and $\slnb(n)$ commute.
Hence the claim.
\end{proof}
\medskip

This proposition implies that
$\ga(x_1+y_1,x_2+y_2)=\ga(x_1,x_2)+\ga(y_1,y_2)$ for $x_1,x_2\in\laxs(n)$ and
$y_1,y_2\in\laxsl(n)$. Hence, $\ga=\ga'\oplus\ga''$.
If $\ga$ is local and/or $\L$-invariant the same is true for 
$\ga'$ and $\ga''$.

First we consider the algebra $\laxs(n)$. It is isomorphic
to $\A$, the isomorphism is given by
\begin{equation}
\laxs(n)\cong \A, \qquad
L\mapsto  \frac 1n\;\tr(L).
\end{equation}  
In \cite[Thm. 4.3]{Scocyc}
it was shown that up to rescaling the unique $\L$-invariant cocycle for
$\A$ is given by 
\begin{equation}
\gamma_{\A}(f,g)=\cins fdg=\res_{P_+}(fdg)
\end{equation}
 (here $C_S$ is a circle around the point $P_+$)

Hence,
\begin{equation}
\ga'(L,M)=\alpha\res_{P_+}(\tr(L)\cdot \tr(dM))
=\alpha\ga_2(L,M),
\end{equation}
by Definition \refE{ga2}.

For the cocycle $\ga''$ of $\laxsl(n)$ we use \refP{lsimp} and
obtain $\ga''=\beta\gamma_{1,\omega}$. Altogether we showed
\begin{proposition}\label{P:lcom}
\begin{equation}
\dim\H_{loc,\L}(\laxgl(n),\C)=2.
\end{equation}
A basis is given by the classes of $\gamma_{1,\omega}$
and $\gamma_2$.
Moreover, every $\L$-invariant cocycle which is bounded from above is local.
\end{proposition}

\section{Uniqueness of the cohomology class for the simple case}
\label{S:direct}
By a quite different approach we will show in
this section 
that for a simple Lie algebra the space 
of local cohomology classes is at most one-dimensional.
We will not require  $\L$-invariance a priori. 
Combining this result with the result of the last section that
 for a simple Lie algebra the space of 
$\L$-invariant local cohomology classes is one-dimensional 
we see that in the simple case each local cohomology class is automatically
an $\L$-invariant cohomology class. Moreover, we showed there that 
it has a unique $\L$-invariant representing cocycle which is 
given as  a multiple of $\ga_{1,\w}$.

\begin{theorem}\label{T:locuni}
Let $\g$ be a finite-dimensional simple classical Lie algebra 
over $\C$ and
$\gb$ the associated infinite-dimensional
Lax operator  algebra with its almost-grading.
Every local cocycle on  $\gb$ is cohomologous up to rescaling 
to a uniquely defined cocycle which is bounded from above
by zero. In particular, the space of local cohomology classes
is at most one-dimensional and up to equivalence and rescaling
there is at most one non-trivial local cohomology class.
\end{theorem}
\begin{remark}
We will even show the following. Let $\g$ be a simple finite-dimensional
Lie algebra and $\gb$ any associated two-point algebra 
of {\it current type}, e.g. 
a Lax operator algebra, a Krichever-Novikov current algebra
$\g\otimes \A$, a loop  algebra $\g\otimes\C[z,z^{-1}]$,
then every cocycle  bounded from above is
cohomologous to a cocycle which is fixed by 
its value at  one special pair of elements in $\gb$
(i.e. by $\ga(H^\a_1,H^{\a}_{-1})$ for one fixed simple root $\a$, see below
for the notation).
Hence in these cases the cohomology spaces are  at most  1-dimensional.
Besides the structure of $\g$ we only use the almost-gradedness of
$\gb$ with leading terms given in \refE{chevlax}.
\end{remark}

First let us recall the following facts about the Chevalley generators
of $\g$.
Choose a root space decomposition
$\g=\fh\oplus_{\alpha\in\Delta}\g^\alpha$. 
As usual $\Delta$ denotes the set of all roots
$\alpha\in\fh^*$. Furthermore, let
$\{\alpha_1,\alpha_2,\ldots,\alpha_p\}$
be a set of simple roots ($p=\dim\fh$).
With respect to this basis, all roots split into positive and negative
roots, $\Delta_+$ and $\Delta_-$
respectively. With $\a$  a positive root, $-\a$ is a negative root and
vice versa.
For $\a\in\Delta$ we have $\dim\g^\alpha=1$.
Certain elements $E^\a$, $\a\in\Delta$ and $H^\a\in \fh$ can be fixed so
that for every positive root $\a$
\begin{equation}\label{E:chev}
[E^\a,E^{-\a}]=H^\a,\qquad
[H^\a,E^{\a}]=2E^{\a},\qquad
[H^\a,E^{-\a}]=-2E^{-\a}.
\end{equation}
We use also $H^i:=H^{\a_i}$, $i=1,\ldots,p$ for the elements
assigned to the simple roots.
A vector space basis, the Chevalley basis, of $\g$ is given by
$\{E^\a, \alpha\in\Delta;\  H^i, 1\le i\le p\}$.

Denote by $(\,\, ,\,\,)$ the product on
$\fh^*$ induced by  the Cartan-Killing form of $\g$.
We have the additional relations
\begin{equation}\label{E:chevfull}
\begin{aligned}\
[H^\a,H^\b]&=0,\quad
\\
[H^\a,E^{\pm\b}]&=\pm2\frac{(\b,\a)}{(\b,\b)}E^{\pm \a},
\\
[H,E^{\a}]&=\a(H)E^\a,\quad H\in\fh,
\\
[E^\a,E^{\b}]&=
\begin{cases}
H^\a, &\a\in\Delta_+,\ \b=-\a,
\\
-H^\a, &\a\in\Delta_-,\ \b=-\a,
\\
\pm(r+1)E^{\a+\b},&\a,\ \b,\ \a+\b\in\Delta,
\\
0,&\text{otherwise.}
\end{cases}
\end{aligned}
\end{equation}
Here $r$ is the largest nonnegative integer such that $\a-r\b$
still is a root.

\medskip
As in the other parts of this article, we denote by
$E_n^\a$, $H_n^\a$ the corresponding elements in $\laxg$ of degree
$n$ for which the expansions at $P^+$ start with $E^\a z_+^n$ and
$H^\a z_+^n$ respectively.
A basis for $\laxg$ is  given by
\begin{equation}\label{E:laxb}
\{\;E_n^\a, \alpha\in\Delta;\  H_n^i, 1\le i\le p\mid\ n\in\Z\;\}.
\end{equation}
The structure equations, up to higher degree 
terms, are
\begin{equation}\label{E:chevlax}
\begin{aligned}\
[H_n^\a,H_m^\b]&\equiv0,\quad
\\
[H_n^\a,E_m^{\pm\b}]&\equiv\pm2\frac{(\b,\a)}{(\b,\b)}E_{n+m}^{\pm \b},
\\
[H_n,E_m^{\a}]&\equiv\a(H)E_{n+m}^\a,\quad H\in\fh,
\\
[E_n^\a,E_m^{\b}]&\equiv
\begin{cases}
H_{n+m}^\a, &\a\in\Delta_+,\ \b=-\a,
\\
-H_{n+m}^\a, &\a\in\Delta_-,\ \b=-\a,
\\
\pm(r+1)E_{n+m}^{\a+\b},&\a,\ \b,\ \a+\b\in\Delta,
\\
0,&\text{otherwise.}
\end{cases}
\end{aligned}
\end{equation}
Recall that the symbol $\equiv$ denotes equality up to elements of
degree higher than the sum of the
degrees of the elements under consideration.  
Here,  the elements not written down  are elements of degree $>n+m$.
Also recall that by the almost-gradedness
there exists a $K$, independent of $n$ and $m$,
such that only elements of degree $\le n+m+K$
appear.

\medskip

Let $\ga'$ be a  cocycle for $\laxg$ which is bounded from above.
For the elements in $\g$ we get
\begin{equation}
E^{\pm\a}=\pm 1/2[H^\a,E^{\pm\a}],
\quad
H^i=[E^{\alpha_i}, E^{-\alpha_i}], \ i=1,\ldots, p.
\end{equation}
Consequently, for $\laxg$ we obtain
\begin{equation}\label{E:ehn}
\begin{aligned}
E_n^{\pm\a}&=\pm 1/2[H_0^\a,E_n^{\pm\a}]+Y(n,\a),
\\
H_n^i&=[E_0^{\alpha_i}, E_n^{-\alpha_i}]+Z(n,i), \ i=1,\ldots, p.
\end{aligned}
\end{equation}
with elements $Y(n,\a)$ and $Z(n,i)$  which are
sums of elements of degree between
$n+m+1$ and $n+m+K$.
Fix a number  $M\in Z$ such that the cocycle $\ga'$ vanishes for
all levels $\ge M$.
We define a linear map $\Phi:\laxg\to\C$ by
(descending) induction on the degree of the basis elements
\refE{laxb}.
First
\begin{equation}\label{E:ehni}
\Phi(E_n^{\a}):=\Phi(H_n^i):=0,\qquad \a\in\Delta,\ i=1,\ldots,p, \quad n\ge M.
\end{equation}
Next we define inductively ($\a\in\Delta_+$)
\begin{equation}
\begin{aligned}\label{E:bel}
\Phi(E_n^{\pm \a})&:=\pm 1/2\ga'(H_0^\a,E_n^{\pm a})+\Phi(Y(n,\pm\a)),
\\
\Phi(H_n^i)&:=\ga'(E_0^{\a_i},E_n^{-\a_i})+\Phi(Z(n,i)).
\end{aligned}
\end{equation}
The cocycle  $\ga=\ga'-\delta\Phi$ is 
cohomologous to the original cocycle $\ga'$. As $\ga'$ is bounded from
above, and,  by  definition, $\Phi$ is also 
bounded from above, the cocycle $\ga$ is  bounded from above too.

By the construction of $\Phi$ we have
$\Phi([H_0^\a,E_n^{\pm\a}]=\ga'(H_0^\a,E_n^{\pm\a})$ and
$\Phi([E_0^{\a_i},E_n^{-\a_i}])=\ga'(E_0^{\a_i},E_n^{-\a_i})$.
Hence
\begin{proposition}\label{P:p1}
\begin{equation}\label{E:p1}
\ga(H_0^\a,E_n^{\pm\a})=0,\quad
\ga(E_0^{\a_i},E_n^{-\a_i})=0,\qquad \a\in\Delta_+,\
i=1,\ldots,p,\quad n\in\Z.
\end{equation}
\end{proposition}
\begin{definition}\label{D:norc}
A cocycle $\ga$ is called \emph{normalized} if it
fulfills \refE{p1}.
\end{definition}
Above we showed that every cocycle
bounded from above is cohomologous to a normalized one, which is
also bounded from above.
In the following we assume that our cocycle is already normalized.
\begin{proposition}\label{P:p2}
Let $H$ be an arbitrary element of $\fh$ then
\begin{equation}\label{E:p2}
\ga(E_m^\a,H_n)\equiv 0,\quad  \a\in\Delta,\ n,m\in \Z,
\end{equation}
i.e. these values are (universal) expressions of values at higher level.
\end{proposition}
\begin{proof}
We start from the cocycle relation
\begin{equation}
\ga([H_n,H_0^{\a}],E_m^\a)
+
\ga([H_0^{\a},E_m^\a],H_n)
+
\ga([E_m^\a,H_n],H_0^{\a})=0.
\end{equation}
The commutator in the first term is of higher level. Hence
using the relations \refE{chevlax} we obtain
\begin{equation}
\a(H^\a)\ga(E_m^\a,H_n)+\a(H)\ga(E_{m+n}^\a,H_0^{\a})\equiv 0.
\end{equation}
By \refE{p1} the last term vanishes. As $\a(H^\a)\ne 0$
the claim follows.
\end{proof}
\begin{proposition}\label{P:p3}
Let $\a$ and $\b$ be roots such that  $\b\ne -\a$, then
\begin{equation}\label{E:p3}
\ga(E_m^\a,E_n^\b)\equiv 0,\quad  n,m\in \Z,
\end{equation}
i.e. they are (universal) expressions of values at higher level.
\end{proposition}
\begin{proof}
Let $H$ be an arbitrary element of $\fh$.
Again we start from the cocycle relation
\begin{equation}
\ga([E_m^\a,H_0],E_n^\b)+
\ga([H_0,E_n^\b],E_m^\a)+
\ga([E_n^\b,E_m^\a],H_0)=0.
\end{equation}
Here, the third term is of higher level. If
$\a+\b\in\Delta$ this follows from \refE{p2}. If
$\a+\b\notin\Delta$ then $[E^\a,E^\b]=0$ and the degree 
of $[E^\a_n,E^\b_m]$ is bigger than $m+n$.

For the first two terms we find  (using
\refE{chevlax})
\begin{equation}
(\a+\b)(H)\ga(E_n^\b,E_m^\a)\equiv 0.
\end{equation}
As we can choose $H$ such that $(\a+\b)(H)\ne 0$ we get the
claim.
\end{proof}

Consider the cocycle relation
\begin{equation}
\ga([E_0^\a,E_n^{-\a}],H_m^\b)+
\ga([E_n^{-\a},H_m^\b],E_0^\a)+
\ga([H_m^\b,E_0^\a],E_n^{-\a})=0.
\end{equation}
Using \refE{chevlax} and ignoring higher levels we obtain
for positive roots $\a$ and $\b$
\begin{equation}\label{E:ct1}
\ga(H_n^\a,H_m^\b)+2\frac {(\a,\b)}{(\a,\a)}
\left(\ga(E_{n+m}^{-\a},E_0^\a)+
\ga(E_{m}^{\a},E_n^{-\a})\right)\equiv 0.
\end{equation}
\begin{proposition}\label{P:ct3}
Let $\a$ be a simple root, then
\begin{equation}
\ga(E_n^{-\a},E_m^{\a})\equiv \frac 12\ga(H_n^\a,H_m^\a).
\end{equation}
\end{proposition}
\begin{proof}
Take the same simple root for $\a$ and $\beta$ in \refE{ct1}. By
\refP{p1} 
\newline 
$\ga(E^{-\a}_{n+m},E_0^\a)=0$ and the claim follows.
\end{proof}
Combining \refE{ct1} and \refP{ct3} we obtain
\begin{equation}\label{E:ct2}
\ga(H_n^\a,\H_m^\b)\equiv\frac {(\a,\b)}{(\a,\a)}
\ga(H_n^\a,\H_m^\a)
\end{equation}
for a 
simple root $\a$ and an arbitrary root $\b$.
\begin{proposition}\label{P:p11}
Let $\a$ be a positive root and $\alpha_1$ a simple root
such that $\a+\a_1$ is again a root then
\begin{equation}
\ga(E_m^{\a+\a_1},E_n^{-(\a+\a_1)})\equiv
s_{\a,\a_1}\cdot \ga(E_m^{\a},E_n^{-\a}),
\end{equation}
with a constant $s_{\a,\a_1}\ne 0$.
\end{proposition}
\begin{proof}
We consider the cocycle relation
\begin{equation}
\ga([E_m^{\a+\a_1},E_0^{-\a_1}],E_n^{-\a})+
\ga([E_0^{-\a_1},E_n^{-\a}],E_m^{\a+\a_1})+
\ga([E_n^{-\a},E_m^{\a+\a_1}],E_0^{-\a_1})=0.
\end{equation}
As $\a_1$ is a simple root we can apply \refP{p1} and see that the
third term is of higher level. For the first two terms we use
\refE{chevlax}, namely the Chevalley relation involving $r$.
Since $r+1\ne 0$
in \refE{chevlax}, the claim follows.
\end{proof}
\begin{proposition}\label{P:p22}
Let $\alpha$ and $\beta$ be two simple roots. Then
\begin{equation}\label{E:2sn}
\ga(H_n^\a,H_m^\a)\equiv \frac {(\a,\a)}{(\b,\b)}
\ga(H_n^{\b},H_{m}^\b).
\end{equation}
\end{proposition}
\begin{proof}
Let $\a$ and $\b$ be two simple roots. By \refE{ct2}
\[ \ga(H_n^\a,H_m^\b)\equiv
\frac {(\a,\b)}{(\a,\a)} \ga(H_n^\a,H_m^\a)
\]
and similarly
\[  \ga(H_m^\b,H_n^\a)\equiv
\frac {(\b,\a)}{(\b,\b)} \ga(H_m^\b,H_n^\b).
\]
Since $\ga$ is skew-symmetric and $(.,.)$ is symmetric, we find
\begin{equation}
\frac {(\a,\b)}{(\a,\a)} \ga(H_n^\a,H_m^\a)\equiv
\frac {(\a,\b)}{(\b,\b)} \ga(H_n^\b,H_m^\b).
\end{equation}
If $(\a,\b)\ne 0$ we obtain directly
\refE{2sn}.
If not then by the irreducibility of the root system
we can always find a chain of simple roots $\a^{(j)}, j=0,\ldots, k$ with
 $\a^{(0)}=\a$, $\a^{(k)}=\b$ and $(\a^{(j)},\a^{(j+1)})\ne 0$.
Evaluating the pairwise results along this chain proves the claim.
\end{proof}
\begin{proposition}\label{P:p33}
Let $\a_1$ be a fixed simple root and $\a$ an
arbitrary positive root, then
\begin{equation}
\ga(E_m^{\a},E_n^{-\a})\equiv
s_{\a,\a_1}\cdot \ga(E_m^{\a_1},E_n^{-\a_1})
\equiv
t_{\a,\a_1}\cdot \ga(H_m^{\a_1},H_n^{\a_1}),
\end{equation}
with  constants $s_{\a,\a_1},t_{\a,\a_1},\ne 0$.
\end{proposition}
\begin{proof}
As $\a$ is a positive root it is a non-trivial sum of simple
roots.  Let $\a_2$ be  one of those.
Repeated application of  
\refP{p11} yields that the 
value $\ga(E_m^{\a},E_n^{-\a})$ can be reduced to
$\ga(E_m^{\a_2},E_n^{-\a_2})$. Combining 
 Propositions
\ref{P:ct3} and \ref{P:p22} gives the claimed dependence on
$\ga(E_m^{\a_1},E_n^{-\a_1})$ and $\ga(H_m^{\a_1},H_n^{\a_1})$.
\end{proof}
\begin{proposition}\label{P:p44}
Fix  a simple root $\a_1$ and let $\a$ and  
$\beta$  be arbitrary roots then
\begin{equation}
\ga(H_n^{\a},H_{m}^\beta)\equiv
s_{\a,\b}\cdot \ga(H_n^{\a_1},H_{m}^{\alpha_1}), \quad \text{with }
s_{\a,\b}\in\C.
\end{equation}
\end{proposition}
\begin{proof}
As the $H^{\a_i}$, $i=1,\ldots,p$ form  a basis of
the Cartan subalgebra $\fg$,   every element
$H^\a$ is a linear combination of them. This extends to the
elements $H^\a_n$. By the bilinearity of the cocycle, 
\refP{p22} and Equation \refE{ct2} the claim follows.
\end{proof}
Let us summarize the results obtained in Propositions
\ref{P:p1}, \ref{P:p3}, \ref{P:p33}, and \ref{P:p44}.
\begin{proposition}\label{P:summary}
Let $\a_1$ be a fixed simple root and $\ga$ the
above defined cocycle, then for all $n,m\in\Z$
\begin{equation}
\begin{aligned}
\ga(E_m^\a,H_n)&\equiv 0, \qquad\qquad  H\in\fh,\  \a\in\Delta
\\
\ga(E_m^\a,E_n^\b)&\equiv 0, \qquad \qquad \a,\b\in\Delta, \ \b\ne-\a,
\\
\ga(E_m^\a,E_n^{-\a})&\equiv s\ga(H_m^{\a_1},H_n^{\a_1}), \qquad
\a\in\Delta,
\\
\ga(H_m^\a,H_n^{\b})&\equiv
t\ga(H_m^{\a_1},H_n^{\a_1}), \qquad \a,\b\in\Delta_+,
\end{aligned}
\end{equation}
with $s,t\in\C$.
\end{proposition}

\bigskip
Next we consider for a simple  root $\a$ the relation
\begin{equation}
\ga(H_m^\a,[E_n^\a,E_k^{-\a}])+
\ga(E_n^\a,[E_k^{-\a},H_m^\a])+
\ga(E_k^{-\a},[H_m^\a,E_n^\a])=0.
\end{equation}
Using \refE{chevlax} we obtain
\begin{equation}
\ga(H_m^\a,H_{n+k}^\a)+
\ga(E_n^\a,2E_{k+m}^{-\a})+
\ga(E_k^{-\a},2E_{m+n}^\a)\equiv0.
\end{equation}
As the root is simple, we can use \refP{ct3} and obtain
the important relation
\begin{equation}\label{E:hrec}
\boxed{\ga(H_m^\a,H_{n+k}^\a)+
\ga(H_n^\a,H_{k+m}^\a)+
\ga(H_k^\a,H_{m+n}^\a)\equiv0.}
\end{equation}

\begin{proposition}\label{P:p5}
Let $\a$ be a simple root then we have
\begin{equation}\label{E:hn0}
\ga(H_n^\a,H_0^\a)\equiv 0,
\end{equation}
and
\begin{equation}\label{E:hnrec}
\ga(H_{n+1}^\a,H_{l-(n+1)}^\a)\equiv
\ga(H_{n-1}^\a,H_{l-(n-1)}^\a)+
2\ga(H_{1}^\a,H_{l-1}^\a).
\end{equation}
\end{proposition}
\begin{proof}
Take the values $m=k=0$ in \refE{hrec}, then  the claim
\refE{hn0} follows from
the antisymmetry.
Setting $m=-1$ and $k=l-n+1$ in \refE{hrec}
we obtain
\begin{equation}\label{E:eqa}
\ga(\ha_{-1},\ha_{l+1})+
\ga(\ha_{n},\ha_{l-n})+
\ga(\ha_{l-(n-1)},\ha_{n-1})\equiv 0.
\end{equation}
With $m=1$ and $k=l-n-1$ we get
\begin{equation}\label{E:eqb}
\ga(\ha_{1},\ha_{l-1})+
\ga(\ha_{n},\ha_{l-n})+
\ga(\ha_{l-(n+1)},\ha_{n+1})\equiv 0.
\end{equation}
Subtracting \refE{eqa} from \refE{eqb} yields
\begin{equation}\label{E:eqc}
\ga(\ha_{l-(n+1)},\ha_{n+1})\equiv
\ga(\ha_{l-(n-1)},\ha_{n-1})-
\ga(\ha_{1},\ha_{l-1})
+
\ga(\ha_{-1},\ha_{l+1}).
\end{equation}
By setting $n=-m$ and $k=l$ in \refE{hrec} we get
\begin{equation}\label{E:eqc1}
\ga(\ha_{-n},\ha_{n+l})+
\ga(\ha_{n},\ha_{l-n})+
\ga(\ha_{l},\ha_{0})\equiv 0.
\end{equation}
The last term does not contribute by \refE{hn0}. Hence
\begin{equation}\label{E:eqd}
\ga(\ha_{n},\ha_{l-n})\equiv
-\ga(\ha_{-n},\ha_{l+n}).
\end{equation}
If we plug \refE{eqd} into \refE{eqc} and use the
antisymmetry we obtain \refE{hnrec}.
\end{proof}
\begin{proposition}\label{P:p51}
Let $\a$ be a simple root. At level $l=0$
the cocycle values are given by
the relations
\begin{equation}\label{E:cat0}
\ga(H_n^{\a},H_{-n}^\a)\equiv n\cdot
\ga(H_1^{\a},H_{-1}^\a),\quad \ga(H_0^{\a},H_0^{\a})=0.
\end{equation}
\end{proposition}
\begin{proof}
If we take the value $l=0$ in \refE{hnrec}  we obtain the
relation
\begin{equation}
\ga(H_{n+1}^\a,H_{-(n+1)}^\a)\equiv
\ga(H_{n-1}^\a,H_{-(n-1)}^\a)+
2\ga(H_{1}^\a,H_{-1}^\a),
\end{equation}
which yields the expression as claimed.
\end{proof}
\begin{proposition}\label{P:p6}
For a simple root $\a$ and for a level $l\ne 0$ we have
\begin{equation}
\ga(H_{n}^\a,H_{l-n}^\a)\equiv 0.
\end{equation}
\end{proposition}
\begin{proof}
First let $l>0$. Using the recursion \refE{hnrec} and the result
\refE{hn0} we see that the claim will be true if it is verified for
$\ga(\ha_1,\ha_{l-1})$.
For $l=1$ using \refE{hn0} we get  $\ga(\ha_1,\ha_{0})\equiv 0$.
For $l=2$ by the antisymmetry   $\ga(\ha_1,\ha_{1})=0$.
Hence let $l>2$. We set $k=l-r-1$, $n=1$ and $m=r$ in
\refE{hrec}:
\begin{equation}\label{E:eqe}
\ga(\ha_{1},\ha_{l-1})+\ga(\ha_{r},\ha_{l-r})-
\ga(\ha_{r+1},\ha_{l-(r+1)})\equiv 0.
\end{equation}
Set $m=(l-2)/2$ for $l$ even and $m=(l-1)/2$ for $l$ odd.
If $r$ runs through $1,2,\ldots,m$ we obtain $m$ equations.
The first equation is always 
\begin{equation}
2\ga(\ha_{1},\ha_{l-1})-
\ga(\ha_{2},\ha_{l-2})\equiv 0.
\end{equation}
The structure of the last equation  depends on the parity of $l$.
For $l$ even and $r=m$ the last term of 
\refE{eqe} is 
$\ga(\ha_{l/2},\ha_{l/2})$ which vanishes. For $l$ odd the last term
of \refE{eqe} coincides with the second term. Hence
\begin{equation}
\ga(\ha_{1},\ha_{l-1})+2
\ga(\ha_{(l-1)/2},\ha_{(l+1)/2})\equiv 0.
\end{equation}
In this case we divide it by 2. By summing up all these 
equations we obtain
\begin{equation}
(m+\epsilon)\ga(\ha_{1},\ha_{l-1})\equiv 0,
\end{equation}
with $\epsilon=1$ for $l$ even and 
$\epsilon=1/2$ for $l$ odd. As in any case $(m+\epsilon)>0$ this shows the
claim.
\newline
For $l<0$ we see that the claim is shown if it is true for
$\ga(\ha_{-1},\ha_{l+1})$.
The argument works in the same way as above. For $l=-1,-2$ it
follows immediately. For $l<-2$ we plug  in 
$k=l-r+1$, $n=-1$ and $m=r$ in \refE{hrec}, and obtain
\begin{equation}\label{E:eqf}
\ga(\ha_{-1},\ha_{l+1})+
\ga(\ha_{r},\ha_{l-r})
-
\ga(\ha_{r-1},\ha_{l-(r-1)})
\equiv 0.
\end{equation}
We set $m:=(-l-2)/2$ for $l$ even and $m:=(-l-1)/2$ for $l$ odd and
consider the equation \refE{eqf} for 
$r=-1,-2,\ldots,-m$. They have a  structure  similar to 
the case $l>0$ and 
we can sum them up to obtain the statement about 
$\ga(\ha_{-1},\ha_{l+1})$.
\end{proof}

\medskip
\begin{proof}[Proof of \refT{locuni}]
After adding a suitable coboundary we might replace
the given  $\ga$  by a normalized $\ga$ (see \refD{norc}).
By the series of propositions above we showed that
the expressions at level $l$ of the
cocycle can be reduced to expressions of  level $> l$ and
values $\ga(H_n^\a,H_{l-n}^\a)$.
As long as the level is $>0$,  by \refP{p6} also these
values can be expressed by higher level.
Hence by trivial induction, starting with the boundedness from above, 
we obtain that zero is an upper bound for the level of the
cocycle.
Also it follows that the values at level $l<0$ are fixed
by induction by the values at level zero.
Hence it remains to consider
level zero.
By  Propositions \ref{P:summary}, \ref{P:p51}, and
\ref{P:p6}
everything depends only on
$\ga(H_1^{\a},H_{-1}^\alpha)$ for one (fixed) simple root.
Hence the claim follows.
\end{proof}

\begin{proposition}
If a normalized cocycle $\ga$ is a coboundary then it 
vanishes identically.
\end{proposition}
\begin{proof}
As explained above, a normalized cocycle is fixed by the value  
$\ga(H_1^{\a},H_{-1}^\alpha)$. But 
$H_{(1)}^\a:=H_0^\a A_1\equiv H^\a_1$ and 
$H_{(-1)}^\a:=H_0^\a A_{-1}\equiv H^\a_{-1}$. Hence
\begin{equation}
[H_{(1)}^\a,H_{(-1)}^\a]=[H_0^\a,H_0^\a]A_1A_{-1}=0.
\end{equation}
As the cocycle vanishes for positive levels, and as 
$\ga=\delta\phi$ is a coboundary we get 
\begin{equation}
\ga(H_1^{\a},H_{-1}^\alpha)=
\ga(H_{(1)}^{\a},H_{(-1)}^\alpha)
=\phi([H_{(1)}^\a,H_{(-1)}^\a])=\phi(0)=0.
\end{equation}
Hence, all cocycle values are zero, as claimed.
\end{proof}

\begin{remark}
In the classical case $\gb=\g\otimes\C[z,z^{-1}]$
the algebra is graded. Hence there are no
higher order terms in \refE{ehn}, and we can even start
with an arbitrary cocycle, not necessarily
bounded, and take as coboundary the one defined
via \refE{bel}. As all our $\equiv$ symbols
are replaced by $=$ symbols there are nowhere higher
contributions, and we obtain the same uniqueness result
as above.
In this very special case the presented chain of arguments 
simplifies and is then 
similar to that  of Garland \cite{Gar}.
\end{remark}
 
\medskip
\begin{remark}
A closer look at the arguments used in \refS{induction} and
\refS{direct} shows that we only use (1) the property of
almost-grading of $\gb$ as expressed in \refT{almgrad},  (2)
that there exists a connection $\omega$, which is holomorphic at
$P_+$ with possibly  poles at $P_-$ and  at the
points of weak singularities, such that 
$\gb$ becomes a Lie module over $\L$ with respect to the connection
$\nabla^{(\omega)}$,  and (3) that the  cocycle \refE{g1} is local
with respect to the almost-grading.
Already from these conditions  \refP{almgrad} follows and all
arguments go through for  any suitable definition of $\gb$
associated to a simple Lie algebra $\g$.
\end{remark}
\appendix
\section{Calculations for $\sob(n)$ and $\spnb(2n)$}\label{S:action}
In this appendix we show \refP{action} for
$\gb=\sob(n)$ and $\gb=\spnb(2n)$.
In fact, it only remains to show that for $L\in\gb$ 
we have $\nabla_e^{(\w)}L\in \gb$.
More precisely, we have to verify whether the conditions at the
points  $\gamma_s$ of 
 weak singularities with $\a_s\ne 0$ are fulfilled.
To simplify notation we omit the index $s$ and use $z$ for $z_s$.

\subsection{The case $\gb=\sob(n)$}
$ $

Let $L\in \sob(n)$ given at the weak singularities by the
expansion \refE{glexp} with the conditions \refE{sodef}. 
{}Furthermore, let $\w$ be a connection fulfilling \refE{sodefc}.
The first term in the connection applied to $L$ calculates as
\begin{equation*}
\frac {dL}{dz}=\frac {-L_{-1}}{z^2}+L_1+\sum_{k\ge 1}(k+1)L_{k+1}z^k.
\end{equation*}
For the second term $[\w,L]$ we consider its degree expansion.

\medskip
\noindent
{\it Term of order -2:}
It comes with the matrix coefficient
$$
[\w_{-1},L_{-1}]=
[\a\tb^t-\tb\a^t,\a\b^t-\b\a^{t}].
$$
Using 
$$
\tb^t\a=\a^{t}\tb=1,\quad
\a^t\a=0,\quad \b^t\a=\a^t\b=0,\quad
\epsilon:=\tb^t\b,
$$
we calculate
$$
[\a\tb^t,\a\b^t]=\a\b^t,
\quad
[\a\tb^t,\b\a^t]=\epsilon\a\a^t,
\quad
[\tb\a^t,\a\b^t]=-\epsilon\a\a^t,
\quad
[\tb\a^t,\b\a^t]=-\b\a^t.
$$
Hence 
$$
[\w_{-1},L_{-1}]=\a\b^t-\b\a^t=L_{-1},
$$
and there is no term of order -2.

\medskip
\noindent
{\it Term of order -1:}
Here we have to show that we can write the matrix coefficient
as $\a\hat\beta^t-\hat\b\a^t$ with $\hat \b$ where $\hat\b^t\a=0$.
\begin{align*}
[\w_{-1},L_0]&=[\a\tb^t,L_0]-[\tb\a^t,L_0]=
        \a\tb^tL_0-L_0\a\tb^t-\tb\a^tL_0+L_0\tb\a^t
\\
   &=\a(\tb^tL_0-\ka\tb^t)+(\ka\tb+L_0\tb)\a^t
\\
[L_{-1},\w_0]&=
\a(\b^t\w_0-\tilde\ka\b^t)+(\tilde\ka\b+\w_0\b)\a^t.
\end{align*}
If we set
$$
\hat\b=-(\ka\tb+\tilde\ka\b+L_0\tb+\w_0\b)
$$
we obtain
$$
[\w_{-1},L_0]+[L_{-1},\w_0]=\a\hat\beta^t-\hat\b\a^t.
$$
Furthermore,
$$
\hat\b^t\a=
\tb^tL_0\a+\b^t\w_0\a-\ka\tb^t\a-\tilde\ka\b^t\a=0,
$$
as $\a$ is an eigenvector of $L_0$ with eigenvalue $\ka$, and 
of  $\w_0$ with eigenvalue $\tilde\ka$.

\medskip
\noindent
{\it Zero order:}
Here we have to show that there exists $\hat\ka\in\C$ such that
$$
([\w_{-1},L_{1}]+[\w_{0},L_{0}]+[\w_{1},L_{-1}]+L_1)\a=\hat\ka\a.
$$
The second term  vanishes. Further on, 
\begin{align*}
[\w_{-1},L_{1}]\a&=
      \a\tb^tL_1\a-\tb\a^tL_1\a-L_1\a\tb^t\a+L_1\tb\a^t\a
\\
&=-L_1\a+\mu\a,
\end{align*}
with $\mu=\tb^tL_1\a$. Note that $\a^tL_1\a=0$ as
$L_1$ is skew-symmetric. The last term also vanishes due to 
$\a^t\a$. Also
$$
[\w_{-1},L_{1}]\a=\mu'\a,
$$
with $\mu'=-\b^t\w_1\a$. Hence we get the 
proclaimed eigenvalue property.
\qed


\subsection{The case $\gb=\spnb(2n)$}
$ $

Let $\w$ be a connection form as described by 
\refE{connl} and \refE{spdefc}.
For the convenience of the reader 
we write down the above listed (see \refS{alg})
conditions again
\begin{gather}
\w_{-1}= (\a{\tilde\b}^t+{\tilde\b}\a^t)\s ,\label{E:i2c}
\\
\label{E:j2c}  {\tilde\b}^t\s\a=1,
\\
\label{E:k2c}  \w_0\a={\tilde\k}\a ,
\\
\label{E:l2c}
\a^t\s\w_1\a=0.
\end{gather}
The corresponding conditions for $L\in\spnb(2n)$ are
\begin{gather}
\label{E:i''}
 L_{-2}=\nu\a\a^t,\ \ L_{-1}=(\a\b^t +\b\a^t)\s,
\\
\label{E:j''}
 \b^t\s\a=0,
\\
\label{E:k''}
L_0\a=\k\a
\\
\label{E:l''}
  \a^t\s L_1\a=0.
\end{gather}
We have
\[ \frac{dL}{dz}=-2\nu\frac{\a\a^t\s}{z^3}-\frac{(\a\b^t+\b\a^t)\s}{z^2}
   +L_1+2L_2z\ldots\ ,
\]
\[
  [\w,L]=\frac{[\w_{-1},L_{-2}]}{z^3}+\frac{[\w_{-1},L_{-1}]
         +[\w_0,L_{-2}]}{z^2}+
   \frac{[\w_{-1},L_0]+[\w_0,L_{-1}]+[\w_1,L_{-2}]}{z}+\ldots\ .
\]
Let 
$$L'dz=\nabla^{(\w)}L=\left(\frac {dL}{dz}+[\w,L]\right)dz.
$$
We calculate the matrix coefficients of $L'$.

\medskip
\noindent
{\it Term of order -3:} 
\begin{align*}
  L'_{-3} &=
  -2\nu\a\a^t\s+[(\a{\tilde\b}^t+{\tilde\b}\a^t)\s,\nu\a\a^t\s]\\
  &=-2\nu\a\a^t\s+\nu\a({\tilde\b}^t\s\a)\a^t\s-\nu\a(\a^t\s{\tilde\b})\a^t\s.
\end{align*}
By \refE{j2c} and skew-symmetry of $\s$ we obtain $L_{-3}=0$.

\medskip
\noindent
{\it Term of order -2:} 
\begin{align*}
  L'_{-2}&=-(\a\b^t+\b\a^t)\s
          +[(\a{\tilde\b}^t+{\tilde\b}\a^t)\s,(\a\b^t+\b\a^t)\s]
          +\nu[\w_0,\a\a^t\s]\\
     &=-(\a\b^t+\b\a^t)\s + \a({\tilde\b}^t\s\a)\b^t\s+
        \a({\tilde\b}^t\s\b)\a^t\s
       -\a(\b^t\s{\tilde\b})\a^t\s-\b(\a^t\s{\tilde\b})\a^t\s \\
      &\qquad\qquad\qquad\qquad+2\nu{\tilde\k}\a\a^t\\
      &=2({\tilde\b}^t\s\b+\nu{\tilde\k})\a\a^t\s .
\end{align*}
Hence, $L'_{-2}$ is of the required form \refE{i''}.

\medskip
\noindent
{\it Term of order  -1:} 
\begin{align*}
  L'_{-1}& = [(\a{\tilde\b}^t+{\tilde\b}\a^t)\s,L_0]
            +[\w_0,(\a\b^t+\b\a^t)\s]+[\w_1,\nu\a\a^t\s]\\
        &=[\a{\tilde\b}^t\s,L_0]+[{\tilde\b}\a^t\s,L_0]
            +[\w_0,\a\b^t\s]+[\w_0,\b\a^t\s]
            +[\w_1,\nu\a\a^t\s]\\
        &=\a({\tilde\b}^t\s L_0)-\k\a{\tilde\b}^t\s-\k{\tilde\b}\a^t\s
          -L_0{\tilde\b}\a^t\s\\
        &\qquad \qquad
             +{\tilde\k}\a\b^t\s-\a(\b^t\s\w_0)+
             \w_0\b\a^t\s+{\tilde\k}\b\a^t\s\\
        &\qquad\qquad \qquad+\nu(\w_1\a)\a^t\s-\nu\a\a^t\s\w_1\\
        &=\a({\tilde\b}^t\s L_0)-\k\a{\tilde\b}^t\s+{\tilde\k}\a\b^t\s
        -\a(\b^t\s\w_0)-\nu\a\a^t\s\w_1\\
        &\qquad\qquad -\k{\tilde\b}\a^t\s-L_0{\tilde\b}\a^t\s
        +\w_0\b\a^t\s+{\tilde\k}\b\a^t\s
        +\nu(\w_1\a)\a^t\s\\
      &=\a({\tilde\b}^t\s L_0-\k{\tilde\b}^t\s+{\tilde\k}\b^t\s
        -\b^t\s\w_0-\nu\a^t\s\w_1)\\
       &\qquad\qquad +(-\k{\tilde\b}-L_0{\tilde\b}+\w_0\b+{\tilde\k}\b
        +\nu\w_1\a)\a^t\s.
\end{align*}
Denote the second bracket in the last expression by $\b'$. Then,
by the symplecticity relation $\w_1^t=-\s\w_1\s^{-1}$, and the 
corresponding 
relations for $\w_0$ and $L_0$, we find that the first bracket is
equal to $\b'^t\s$ there, hence
\[
  L'_{-1}=(\a\b'^t+\b'\a^t)\s
\]
as required by the relation \refE{i''}.

It remains to show that $\b'^t\s\a=0$. The expression for
$\b'^t\s$ is exactly given by the just mentioned first bracket. We
have
\begin{align*}
  \b'^t\s\a &= ({\tilde\b}^t\s L_0-\k{\tilde\b}^t\s+{\tilde\k}\b^t\s
        -\b^t\s\w_0-\nu\a^t\s\w_1)\a \\
        &= ({\tilde\b}^t\s L_0\a-\k{\tilde\b}^t\s\a)+({\tilde\k}\b^t\s\a
        -\b^t\s\w_0\a)-\nu\a^t\s\w_1\a .
\end{align*}
This is zero  by \refE{k''}, \refE{k2c} and \refE{l2c}.

\medskip
\noindent
{\it Order zero term:}
We have to show the relation \refE{k''} for $L'$. We have
\[
  L_0'=[\w_{-1},L_1]+[\w_{0},L_0]+[\w_{1},L_{-1}]+[\w_2,L_{-2}]+L_1.
\]
For the first bracket we have
\begin{align*}
  [\w_{-1},L_1]\a &= (\a{\tilde\b}^t+{\tilde\b}\a^t)\s L_1\a
             -L_1(\a{\tilde\b}^t+{\tilde\b}\a^t)\s\a \\
             &= \a({\tilde\b}^t\s L_1\a)+{\tilde\b}(\a^t\s L_1\a)
             -L_1\a({\tilde\b}^t\s\a)+{\tilde\b}(\a^t\s\a).
\end{align*}
The second and the fourth terms vanish by \refE{l''} and
skew-symmetry of $\s$, respectively. In the third term, we replace
${\tilde\b}^t\s\a$ with $1$. Thus, we obtain
\[
  [\w_{-1},L_1]\a = \a({\tilde\b}^t\s L_1\a)-L_1\a.
\]
Obviously,
\[
  [\w_{0},L_0]\a=0.
\]
For the third bracket, we have
\begin{align*}
  [\w_1,L_{-1}]\a &=[\w_1,(\a\b^t+\b\a^t)\s]\a \\
       &=\w_1\a(\b^t\s\a)+\w_1\b(\a^t\s\a) -
         \a(\b^t\s\w_1\a)-\b(\a^t\s\w_1\a)\\
       &= -\a(\b^t\s\w_1\a).
\end{align*}
At last, for the fourth term we have
\[
    [\w_2,L_{-2}]\a=\nu\w_2\a(\a^t\s\a)-\nu\a(\a^t\s\w_2\a)
\]
where the first term obviously vanishes, and we have
\[
     [\w_2,L_{-2}]\a=-\nu\a(\a^t\s\w_2\a).
\]
Finally, we obtain
\[
   L'_0\a=({\tilde\b}^t\s L_1\a-\b^t\s\w_1\a-\nu(\a^t\s\w_2\a))\a .
\]

\medskip
\noindent
{\it Order one term:}
We have to show \refE{l''} for $L'$: $\a^t\s L'_1\a=0$.
We have
\[
  L'_1= 2L_2+ [\w_{-1},L_2]+[\w_0,L_1]+[\w_1,L_0]+[\w_2,L_{-1}]
        +[\w_3,L_{-2}].
\]
For every bracket $[\cdot,\cdot]$ in this expression,
we calculate the corresponding product $\a^t\s[\cdot,\cdot]\a$
and obtain
\begin{align*}
  \a^t\s[\w_{-1},L_2]\a &=\a^t\s[(\a{\tilde\b}^t +
  {\tilde\b}\a^t)\s,L_2]\a\\
  &= \a^t\s(\a{\tilde\b}^t+{\tilde\b}\a^t)\s L_2\a
  -\a^t\s L_2(\a{\tilde\b}^t+{\tilde\b}\a^t)\s\a     \\
  &= (\a^t\s\a){\tilde\b}^t\s L_2\a+(\a^t\s{\tilde\b})\a^t\s L_2\a
  -\a^t\s L_2\a({\tilde\b}^t\s\a)-\a^t\s L_2{\tilde\b}(\a^t\s\a)
  \\
  &= -2\a^t\s L_2\a
\end{align*}
(by the relations $\a^t\s\a=0$ and \refE{j2c}). The result will
cancel with the one coming from the first term $2L_2$
in the above expression for $L_1'$.

Further on,
\begin{align*}
  \a^t\s[\w_0,L_1]\a &=(\a^t\s\w_0)L_1\a-\a^t\s L_1(\w_0\a) \\
   &= -2{\tilde\k}(\a^t\s L_1\a)\\
   &=0,
\\
  \a^t\s[\w_1,L_0]\a &= \a^t\s\w_1)L_0\a)-(\a^t\s L_0)\w_1\a \\
   &= 2\k(\a^t\s\w_1\a) \\
   &= 0,
\\
  \a^t\s[\w_2,L_{-1}]\a &= \a^t\s\w_2(\a\b^t+\b\a^t)\s\a
                           - \a^t\s(\a\b^t+\b\a^t)\s\w_2\a \\
   &= \a^t\s\w_2\a(\b^t\s\a)+\a^t\s\w_2\b(\a^t\s\a)
       - (\a^t\s\a)\b^t\s\w_2\a-(\a^t\s\b)\a^t\s\w_2\a  \\
   &= 0,
\\
  \a^t\s[\w_3,L_{-2}]\a &= \nu\a^t\s\w_3 \a(\a^t\s\a)
              -\nu(\a^t\s\a)\a^t\s\w_3\a \\
              &=0.
\end{align*}
Hence
\[ \a^t\s L'_1\a=0.
\]
\qed


\end{document}